\documentclass[11pt,reqno,letterpaper]{amsart}  
\usepackage{hyperref}
\usepackage{latexsym,amsmath}
\usepackage{times}
\usepackage[shortlabels]{enumitem}
\usepackage{amsfonts}
\usepackage{amssymb}
\usepackage{mdframed}
\usepackage{latexsym}
\hypersetup{colorlinks=true,citecolor=blue, linkcolor=blue,
urlcolor=blue}
\usepackage{thm-restate}
\usepackage{multicol}
\usepackage{xcolor}
\usepackage{dsfont}

\usepackage{graphicx}
\usepackage[style=trad-abbrv]{biblatex}
\DeclareFieldFormat{titlecase}{#1}

\usepackage[T1]{fontenc}
\usepackage{bbm}
\usepackage{color}
\usepackage{fullpage}

\numberwithin{equation}{section}
\newcommand\G{\vec G}
\renewcommand\H{\vec H}

\newcommand\E{\vec E}
\newcommand\F{\vec F}
\newcommand\V{\vec V}

\newcommand\nid{\!}
\newcommand\ord{\mathrm{o}}
\newcommand\dis{\mathrm{d}}
\newcommand\Sord{S_{\mathrm{o}}}
\newcommand\Sdis{S_{\mathrm{d}}}
\newcommand\tSord{\tilde{S}_{\mathrm{o}}}
\newcommand\Omegaord{\Omega_{\G,\mathrm{o}}}
\newcommand\Omegadis{\Omega_{\G,\mathrm{d}}}
\newcommand\piord{\pi_{\G,\mathrm{o}}}
\newcommand\pidis{\pi_{\G,\mathrm{d}}}
\newcommand\nuord{\nu_{\mathrm{o}}}
\newcommand\rhoord{\rho_{\mathrm{o}}}

\newcommand\nudis{\nu_{\mathrm{d}}}

\newcommand\rhodis{\rho_{\mathrm{d}}}
\newcommand\Ex{\mathbb{E}}

\newcommand\mord{\fm_{\mathrm{o}}}
\newcommand\mdis{\fm_{\mathrm{d}}}
\newcommand\Zord{Z_{\mathrm{o}}}
\newcommand\Zdis{Z_{\mathrm{d}}}

\newcommand\sord{\hat\SIGMA_{\mathrm{o}}}
\newcommand\sdis{\hat\SIGMA_{\mathrm{d}}}

\newcommand\emm{\mathrm{e}}

\renewcommand{\vec}[1]{\boldsymbol{#1}}

\newcommand\SIGMA{\vec\sigma}

\newcommand\fm{\mathfrak{m}}

\newcommand\cA{\mathcal{A}}

\newcommand\cC{\mathcal{C}}

\newcommand\cO{\mathcal{O}}
\newcommand\cP{\mathcal{P}}

\newcommand\cG{\mathcal{G}}
\newcommand\cE{\mathcal{E}}

\newcommand\cN{\mathcal{N}}
\newcommand\cQ{\mathcal{Q}}

\newcommand\cS{\mathcal{S}}

\def\cC{{\mathcal C}}
\def\cE{{\mathcal E}}

\newcommand\eps{\varepsilon}

\newcommand{\vecone}{\vec{1}}

\newcommand{\Bin}{{\rm Bin}}

\newcommand\TV[1]{\left\|{#1}\right\|_{\mathrm{TV}}}

\newcommand\cbc[1]{\left\{{#1}\right\}}

\newcommand\brk[1]{\left\lbrack{#1}\right\rbrack}

\newcommand\pr{\mathbb{P}} 
\renewcommand\Pr{\pr}

\newtheorem{definition}{Definition}[section]

\newtheorem{theorem}[definition]{Theorem}
\newtheorem{lemma}[definition]{Lemma}

\newtheorem{corollary}[definition]{Corollary}

\newtheorem*{thmRCdis}{Theorem~\ref{thm:RCdis}}
\newtheorem*{thmRCord}{Theorem~\ref{thm:RCord}}

\newtheorem*{Lemquiet}{Lemma~\ref{lem:quietplanting}}
\newtheorem*{LemPottstoRC}{Lemma~\ref{lem:PottstoRC}}
\newtheorem*{Lemdecay}{Lemma~\ref{lem:2l-decay-bound-for-twopath}}
\newtheorem*{Lemgiantperc}{Lemma~\ref{lem:giantperc}}
\newtheorem*{LemmaTVatmostWB}{Lemma~\ref{lem:wsm-implied-from-wired-boundary}}
\newtheorem*{LemmaWBfromgiantpaths}{Lemma~\ref{lem:no-long-path-outside-giant=>wired}}
\newtheorem*{LemmaOrderedUpperBound}{Lemma~\ref{lem:ordered-less-than-(d-1)^-3}}

\begin{document}

\title{Planting and MCMC Sampling from the Potts model}

\author{Andreas Galanis, Leslie Ann Goldberg, Paulina Smolarova}
\thanks{
For the purpose of Open Access, the author has applied a CC BY public copyright licence to any Author Accepted Manuscript version arising from this submission.}

\address{Andreas Galanis, {\tt andreas.galanis@cs.ox.ac.uk}, University of Oxford, Department of Computer Science, Wolfson Bldg, Parks Rd, Oxford OX1 3QD, UK}

\address{Leslie Ann Goldberg, {\tt leslie.goldberg@cs.ox.ac.uk}, University of Oxford, Department of Computer Science, Wolfson Bldg, Parks Rd, Oxford OX1 3QD, UK}

\address{Paulina Smolarova, {\tt paulina.smolarova@trinity.ox.ac.uk}, University of Oxford, Department of Computer Science, Wolfson Bldg, Parks Rd, Oxford OX1 3QD, UK}
\thispagestyle{empty}
\begin{abstract} 
 We consider the problem of sampling from the ferromagnetic $q$-state Potts model on the random $d$-regular graph with parameter $\beta>0$. A key difficulty that arises in sampling from the model is the existence of a ``metastability'' window $\beta\in (\beta_u,\beta_u')$, where roughly the distribution has two competing modes, the so-called disordered and ordered phases. This  causes classical Markov-chain algorithms to be  slow mixing from worst-case initialisations.  Nevertheless, Helmuth, Jenssen and Perkins (SODA '19) designed a sampling algorithm that works for all $\beta$, when  $d\geq 5$ and $q=d^{\Omega(d)}$,  using polymers and cluster expansion methods; more recently, their analysis technique has been adapted  to show that a Markov chain (random-cluster dynamics) mixes fast when initialised appropriately, in the same regime of $q,d,\beta$.

 Despite these positive algorithmic results, a well-known bottleneck behind  cluster-expansion arguments is that they inherently only work for large $q$, whereas it is widely conjectured that sampling on the random $d$-regular graph is possible for all $q,d\geq 3$. The only result so far that applies to general $q,d\geq 3$ is by Blanca and Gheissari who showed that  the random-cluster dynamics mixes fast in the ``uniqueness'' regime $\beta<\beta_u$ where roughly only the disordered mode exists. For $\beta\geq \beta_u$ however, a second subdominant mode emerges creating bottlenecks and giving rise to correlations which have been hard to handle, especially for small values of $q$ and $d$.

  Our main contribution is   to perform a delicate analysis of the Potts distribution and the random-cluster dynamics that goes beyond the threshold $\beta_u$. We use planting as the main tool, a technique used in the analysis of random CSPs to capture how the space of solutions is correlated with the structure of the random instance. While planting arguments  provide only weak sampling guarantees generically, here we instead combine planting with the analysis of random-cluster dynamics to obtain significantly stronger guarantees. We are thus able to show that the random-cluster dynamics initialised from all-out mixes fast for all integers $q,d\geq 3$  beyond the uniqueness threshold $\beta_u$, all the way to the optimal threshold $\beta_c\in (\beta_u,\beta_u')$ where the dominant mode switches from disordered to ordered. A more involved analysis also applies to the ordered regime $\beta>\beta_c$ where we obtain an algorithm for all $d\geq 3$ and $q\geq (5d)^5$, improving significantly upon the previous range of $q,d$ by Carlson, Davies, Fraiman, Kolla, Potukuchi, and Yap (FOCS'22).
\end{abstract}

\maketitle 
\thispagestyle{empty}
\newpage
\pagenumbering{arabic}
\section{Introduction}\label{Sec_intro}
The Potts model is  a weighted model 
assigning probabilities to
(non-proper) $q$-colourings. The model originated in statistical physics but has since been a central object of study in various contexts; here, we focus on the computational problem of sampling from the  model.

For an integer $q\geq 3$ and a graph $G=(V,E)$, the $q$-state Potts model on $G$ with parameter $\beta$ is a probability distribution $\mu=\mu_{G,q,\beta}$ on $[q]^V$ where $[q]=\{1,\hdots,q\}$. Each configuration $\sigma\in [q]^V$ has weight $\mu(\sigma)=\emm^{\beta m(\sigma)}/Z_G$ where $m(\sigma)$ is the number of monochromatic edges under $\sigma$; the normalising factor $Z_G$ is the partition function. Throughout, we consider the ``ferromagnetic'' case $\beta>0$,  where configurations with many monochromatic edges are favoured in the distribution.

From a computational complexity perspective, sampling from the ferromagnetic Potts model has various twists relative to other similar models and the complexity of the sampling problem is widely open. On general graphs, the problem is \#BIS-hard for any fixed $\beta>0$ \cite{Leslie1}; for graphs of max degree $d$ (where $d$ is a fixed integer), the problem becomes \#BIS-hard   when $\beta> \beta_c$ \cite{BIShard}, where $\beta_c=\ln \frac{q-2}{(q-1)^{1-2/d}-1}$ is known as the ordered/disordered threshold (we will discuss this in more detail shortly). It is also conjectured that the problem admits a polynomial time algorithm when $\beta<\beta_c$ but this is open in general; in fact, many of the standard tools that are used to analyse Markov chains provably fail well below the $\beta_c$ threshold. The quintessential example here is the random $d$-regular graph, which underpins all the relevant phenomena behind this picture. For $dn$ even, we let $\mathbb{G}=\mathbb{G}_{n,d}$  be a graph chosen uniformly at random from the set of all $d$-regular graphs with $n$ vertices. We use ``whp'' as a shorthand for ``with probability $1-o(1)$ as $n$ grows large''.

For a typical random graph $\mathbb{G}$, it is known that  
a sample from the Potts distribution $\mu_{\mathbb{G}}$ is ``disordered'' for $\beta<\beta_c$, and ``ordered'' for $\beta>\beta_c$; roughly, this means that the colours in the former case appear equally often whereas in the second case one of the colours strictly dominates over the rest (see Lemma~\ref{lem:smallgraph} for the formal statement). 
The intricate feature of the Potts model on the random $d$-regular graph is that these competing modes are both present 
as ``local maxima''
throughout a window $(\beta_u,\beta_u')$ containing~$\beta_c$ 
even though only one of them has the vast majority of the mass (except at
$\beta=\beta_c$ itself,  where both modes appear with constant probability). This already poses problems to standard MCMC algorithms such as Glauber and Swendsen-Wang dynamics since the simultaneous presence of the  modes causes exponentially slow mixing from worst-case initialisations, see \cite{RCM-Helmuth2020,coja2023}. At a more conceptual level, any sampling algorithm has to take into account the presence of the other mode  which obliterates from the very start standard analysis tools (e.g., correlation decay/spectral independence).

Helmuth, Jenssen and Perkins \cite{RCM-Helmuth2020} introduced a cluster-expansion technique that allowed them to control more precisely how much a typical sample differs from the corresponding mode, and obtained an algorithm based on the interpolation method for all $\beta>0$ when $d\geq 5$ and $q\geq d^{\Omega(d)}$; their algorithm applies more generally to expander graphs (under some mild conditions). See also the works \cite{UG1,UG2,BIS2} for various refinements of their method. Following a series of developments \cite{SinclairsGheissari2022, RClattice}, this cluster expansion  expansion argument has been converted into an  MCMC algorithm (for large \(d\) and \(q\)) \cite{RCrandom}  using the random-cluster dynamics with appropriate initialisation that avoids the bottlenecks in the distribution (see Section~\ref{sec:RCdynamics} for details). See also \cite{BlaGhe,Blanca2, gheissari2025rapid} for closely related results on the random $d$-regular graph (and lattices) that apply for large $\beta$.

Despite these positive algorithmic developments, the cluster expansion technique that underpins the analysis of these algorithms is inherently ``perturbative'',  relying roughly on controlling how much a typical configuration differs from the max-energy configurations and arguing that the difference is relatively small. In the case of the Potts model, the resulting estimates are  accurate only when $q$ is quite large relative to $d$. By contrast, it is conjectured that sampling on the random regular graph should be possible for all $q,d\geq 3$ (and all $\beta>0$), and hence  more precise tools are needed to settle the picture for small $q,d$. The only result so far that works for general $q,d\geq 3$ is by Blanca and Gheissari \cite{BG} who showed that random-cluster dynamics mixes in $O(n\log n)$ time when $\beta<\beta_u$, for arbitrary initialisation. Note that for $\beta>\beta_u$ the results of \cite{coja2023} give worst-case initialisations that slow down the mixing time of the Glauber dynamics to $\emm^{\Omega(n)}$; more generally, for $\beta>\beta_u$, the presence of the ordered mode imposes correlation phenomena that are hard to handle, especially for small values of $q,d$.

\subsection{Results} Our main contribution is to do a delicate analysis of the Potts distribution and the random-cluster dynamics that goes  beyond the threshold $\beta_u$, using planting. Planting is a technique used in the analysis of random CSPs to capture how the space of solutions is correlated with the structure of the random instance, see \cite{Plant1,Plant2, Plant3} for some applications. In terms of sampling however, there is no recipe for converting planting methods to efficient sampling algorithms, and typically some extra sampling step is needed on top, see for example \cite{Plant1a,Plant1b} for such approaches; even so, the resulting sample usually has some accuracy limitation, i.e., the TV-distance from the target distribution cannot be made arbitrarily small (typically 
it is required to be at least an inverse polynomial in~$n$). Here, we instead use planting as a means to obtain sharp estimates that are used in the analysis of random-cluster dynamics; we thus show fast convergence of the latter, yielding automatically strong sampling guarantees. A different sampling idea that uses  planting in a somewhat related manner has recently been considered in \cite{huang2025weak}, in the context of sampling from spherical spin glasses.

Using this planting framework, we  give an algorithm to sample from the Potts distribution on the random $d$-regular graph $\mathbb{G}$ that covers the optimal range of ``high temperatures'' $\beta<\beta_c$ for all integers $q,d\geq 3$; as we will explain in more detail later the algorithm is running random-cluster dynamics for $O(n\log n)$ steps, initialised appropriately (from the ``empty'' configuration). 
\begin{theorem}\label{thm:main1}
For integers $d,q\geq 3$ and real $\beta\in (0,\beta_c)$, the following holds whp over $\mathbb{G}=\mathbb{G}_{n,d}$.

There is an MCMC algorithm that, on input $\mathbb{G}$ and $\eps\geq \emm^{-\Omega(n)}$, outputs in  $O\big(n\log n\log(\frac{1}{\eps})\big)$ steps a sample $\hat\sigma\in [q]^n$ whose distribution $\hat\mu$ satisfies $\big\|\hat \mu -\mu_{\mathbb{G},q,\beta}\big\|_{\mathrm{TV}}\leq \eps$.
\end{theorem}

Our analysis framework also applies to the ordered regime $\beta>\beta_c$ where we 
improve significantly upon various restrictions from previous works \cite{UG2,UG1, RCM-Helmuth2020, RCrandom}. For example, \cite{RCM-Helmuth2020, RClattice} applied when  $q\geq d^{\Omega(d)}$ (and $d\geq 5$), whereas the  recent work by Carlson et al \cite{UG2} (see also \cite{UG1}) applied when $d\geq d_0$ and $q\geq d^c$ for some large constants $d_0, c$. Using planting, we obtain a sampling result for all $d\geq 3$ when $q$ is roughly larger than $d^5$. We show the following.
\begin{theorem}\label{thm:main2}
For integers $d\geq 3,q\geq (5d)^5$  and real $\beta>\beta_c$, the following holds whp over $\mathbb{G}=\mathbb{G}_{n,d}$.

There is an MCMC algorithm that, on input $\mathbb{G}$ and $\eps\geq \emm^{-\Omega(n)}$, outputs in  $O\big(n\log n\log(\frac{1}{\eps})\big)$ steps a sample $\hat\sigma\in [q]^n$ whose distribution $\hat\mu$ satisfies $\big\|\hat \mu -\mu_{\mathbb{G},q,\beta}\big\|_{\mathrm{TV}}\leq \eps$.
\end{theorem}

\subsection{Proof outline}\label{sec:rf4f4f3} Our algorithms are based on a well-known edge representation of the Potts model, the random-cluster (RC) model, where configurations are edge subsets weighted by the number of edges present and the number of induced components. We run the random-cluster dynamics (the analogue of Glauber dynamics in this edge setting), starting from the empty set of edges in the disordered regime $\beta<\beta_c$ and the full set of edges in the ordered regime $\beta>\beta_c$. The reason for working in the RC representation is that, in contrast to the Potts model, the RC model enjoys certain monotonicity properties that facilitate tracking the evolution of Glauber dynamics.  In the next Preliminaries section we define the RC model and the relevant random-cluster chain.  We also give a more detailed overview of the phases of the Potts model, and how planting gives a handle for the disordered and ordered settings. Most of these parts are largely imported from \cite{coja2023,BIShard}. In Section~\ref{sec:RCmixing}, we  convert the planting for Potts into suitable planting for the RC model, giving the analogue of the ordered and disordered phases in that setting. Then, we state the main contribution of this work which is establishing weak spatial mixing within the disordered and ordered phases (a notion that originated in \cite{SinclairsGheissari2022}). Sections~\ref{sec:disordered-regime} and~\ref{sec:orderedproof} give the technical core of our arguments.

Roughly, in the disordered regime (Section~~\ref{sec:disordered-regime}), the key ingredient is   to show that the dynamics converges to a set of configurations that do not form large components. For $\beta<\beta_u$,  Blanca and Gheissari~\cite{BG} accomplished this by a very careful random-graph revealing process coupled with the evolution of random-cluster dynamics. Using planting, we are able to capture precisely the correlation between the random graph and a sample from stationarity for $\beta<\beta_c$, reducing the study of the component structure to showing that the planted random graph is subcritical. Using the monotonicity of the RC model, this implies that the RC dynamics initialised from the empty set (``all-out'') enjoys the same component structure as the stationary distribution (for exponentially many steps). Having this in place (Lemma~\ref{lem:dispath}), the rest of the proof is more streamlined after some small reworking of certain correlation decay estimates up to the ordered/disordered threshold $\beta_c$  (Lemma~\ref{lem:2l-decay-bound-for-twopath}).

In the ordered regime (Section~\ref{sec:orderedproof}), the planted model is instead  supercritical and a typical configuration has a ``giant'' component, whose presence complicates the arguments significantly. The goal now is roughly to show that bichromatic edges do not typically form large components; more precisely, the key ingredient is to show that every vertex $v$ is surrounded by close-by vertices (each at some distance $\ell<\tfrac{1}{2}\log_{d-1}n$)  that belong to the giant component, a so-called ``wired'' boundary.   Planting here allows us to do a head-on analysis of the probability that, in a typical sample from stationarity, there exists a path of length $\ell$ that does not include any vertex in the giant. We get a tight bound for this using careful first moment arguments and percolation bounds; then, for $q\geq (5d)^5$, a union bound over the paths starting at $v$ gives the existence of the desired wired boundary (in a typical configuration).  Relative to the earlier works \cite{RCM-Helmuth2020,UG1,UG2,RCrandom},  bichromatic components roughly correspond to so-called polymers; our techniques show that a smaller lower bound on the size of the giant component suffices in order to ensure that the polymers are not large.

\subsection{Further Discussion} The planting arguments that we use in this paper follow by careful first and second moment considerations in \cite{BIShard}, which have been carried out for the ferromagnetic Potts model, see also \cite{bapst2017planting} for similar-flavored ``silent'' planting results. The moment analysis has previously been used in \cite{Chen} for sampling independent sets and proper colourings in random bipartite graphs, though there  the approach  ultimately relies on cluster-expansion type of arguments \cite{BIS1,BIS2}. Our more direct planting approach can be carried out for sampling independent sets on random bipartite graphs (that shares similar monotonicity and planting properties) and obtain sharper bounds for sampling using Glauber dynamics. Note however that, as  pointed out in \cite[Lemma 10]{Chen}, it is unlikely that current polymer-based techniques will achieve better asymptotic estimates in terms of the degree $d$. Our technique should also apply to the $G(n,d/n)$ setting (for arbitrary $q,d>1$) provided that the analogous ``silent'' planting is in place; this seems more easily feasible in the disordered regime but more difficult for the ordered regime where the Poisson-tree neighbourhoods are expected to cause bigger variance.

It is also relevant to note that the results we stated earlier  for the cluster-expansion technique and fast mixing of the random-cluster dynamics in the literature apply to non-integer $q$, in the random-cluster representation. It is conceivable that a similar planting to the one we consider here for the Potts model can be carried out for non-integer $q$ in the random-cluster representation, though the related first and second moment computations will likely require careful arguments for the distribution of the number of components in random graph percolation; see \cite{bencs2022random} for some related developments.



\section{Preliminaries}\label{sec:prelims}

\subsection{The configuration model} To obtain a handle on the random $d$-regular $\mathbb{G}=\mathbb{G}_{n,d}$, we work throughout in the configuration model. Formally, for integers $n,d\geq 1$ with $dn$ even, let $[n]\times [d]$ be a set of ``half-edges''; we denote by $G\sim\cG_{n,d}$ a (multi)graph on vertex set $[n]$ whose edges are obtained by taking a uniformly random perfect matching of the half-edges; we add an edge between $u,v\in [n]$ in $G$ whenever the half-edges $(u,i)$ and $(v,j)$ are matched together for some $i,j\in [d]$. For brevity, instead of $G\sim \cG_{n,d}$ we sometimes use $\G$ to denote the random (multi)graph chosen. Any property that holds whp over $\G$ also holds whp over $\mathbb{G}$ (see, e.g., \cite{JLR}), so henceforth we focus  on the configuration model $\G$.

\subsection{The random-cluster dynamics} \label{sec:RCdynamics}

Given a graph~$G=(V,E)$ and real parameters $q>0$ and $p \in [0,1]$, the random-cluster model is a probability distribution $\pi=\pi_{G,q,p}$ on the subsets of $E$. Let $\Omega_G$ be the set of all configurations. Each configuration $F\subseteq E$ has probability
$\pi(F)=q^{c(F)} p^{|F|} (1-p)^{|E\setminus F|} / Z_{RC,G}$ where $c(F)$ is the number of connected components in the graph $(V,F)$, and $Z_{RC,G}$ is a normalisation factor.
For integer $q$ and $p = 1-\emm^{-\beta}$, it is well known that sampling from this probability distribution  is equivalent to sampling from the $q$-state Potts model with parameter $\beta$, see Section~\ref{sec:RCmixing} for details.

Let $\hat p := \frac{p}{(1-p)q + p}$.
The random-cluster (RC) dynamics initialised from a configuration $X_0$ is  a Markov chain $(X_t)_{t\geq 0}$ whose transition $X_t$ to $X_{t+1}$ is as follows:

\begin{enumerate}
    \item First choose an edge \(e\in E\) uniformly at random.
    \item If \(e\) is a cut edge of the graph \((V, X_t\cup \{e\})\), then with probability \(\hat p\), set \(X_{t+1} := X_t\cup\{e\}\). With all remaining probability set \(X_{t+1} := X_t\setminus \{e\}\)
    \item Otherwise, with probability \(p\), set \(X_{t+1} := X_t\cup\{e\}\). With all remaining probability set \(X_{t+1} := X_t\setminus \{e\}\).
\end{enumerate}
It is a standard fact that the distribution of $X_t$ converges to $ \pi_G$, regardless of the initialisation. To work around the slow convergence in the metastability window when $\beta\in (\beta_u,\beta_u')$, we will later consider the RC dynamics starting from the two extreme configurations, the  all-out ($X_0=\emptyset$) and the all-in $(X_0=E)$.

\subsection{Phases of the Potts model}

Let  $\nu=(\nu_1,\hdots,\nu_q)$ be a $q$-dimensional probability vector. For an $n$-vertex graph $G$, let 
\[S_\nu(\theta)=\Big\{\sigma\in[q]^{V(G)}:\sum_{i\in[q]}\big|| \sigma^{-1}(i)|-n\nu_i\Big|\leq \theta n\Big\},\]
where $\theta=\theta(q,d,\beta)>0$ is a small constant which we will suppress from notation. Hence, $S_\nu(\theta)$ consists of these configurations where the colour frequencies are given by the vector $\nu$. Let also
\[\mbox{$Z_{\nu}(G)=\sum_{\sigma\in S_\nu} w_G(\sigma)$ and $\Psi=\Psi_{d,q,\beta}(\nu)=\lim_{\theta\downarrow 0}\lim_{n\rightarrow \infty} \tfrac{1}{n}\ln\Ex[Z_\nu(\G)]$,}\]
where $w_G(\sigma) = e^{\beta m(\sigma)}$ is the weight of $\sigma$.
In short, the function $\Psi(\nu)$ captures the (logarithm of the) expected contribution to the partition function $Z(\G)$ of the random graph $\G$ given by the configurations 
 in $S_\nu$. As we shall see shortly, the local/global maximisers of $\Psi$ play a key role in determining the modes of the Gibbs distribution.

 There are three relevant thresholds that come into play
 \[\beta_u(q,d)<\beta_c(q,d)<\beta'_u(q,d).\]
 For high temperatures ($\beta<\beta_u$), there is a unique global maximiser of $\Psi$, given by the uniform vector \[\nudis=\big(\tfrac{1}{q},\hdots, \tfrac{1}{q}\big)\]
which corresponds to the ``disordered'' configurations. As the temperature decreases and $\beta$ crosses the threshold $\beta_u$, another local maximum of $\Psi$ emerges at an ``ordered'' vector $\nuord$ where one of the $q$ colours dominates over the others (which are roughly balanced). More precisely, $\nuord$ is given by\footnote{\label{fn:defbetau} The value of $\beta_u$ is defined as the infinimum value of $\beta>0$ such that there exists $a\in (1/q,1)$ which satisfies \eqref{eq:nuord}; then it is simple to check that such an $a$ exists for all $\beta\geq \beta_u$. For certain values of $\beta$ (namely when $\beta\in (\beta_u,\beta_u')$), there may be more than one values of $a\in (1/q,1)$ that satisfy the equation; when this is the case, the value of $a$ that is relevant for the ordered phase is the largest such value.} \begin{equation}\label{eq:nuord}\nuord=\big(a,\tfrac{1-a}{q-1},\hdots, \tfrac{1-a}{q-1}\big), \mbox{ and $a\in (\tfrac{1}{q},1)$ satisfies $a=\frac{t^{d}}{t^{d}+q-1}$ where $\emm^{\beta}-1=\frac{(t-1)(t^{d-1}+q-1)}{t^{d-1}-t}$.}
 \end{equation}
Interestingly, the local maximum at $\nuord$ does not become global until $\beta\geq \beta_c$ where $\beta_c=\ln\frac{q-2}{(q-1)^{1-2/d}-1}$ is known as the ordered/disordered threshold; notably, the only point where both $\nu_\ord$ and $\nu_\dis$ are global maxima is at $\beta=\beta_c$. Note also that $\nudis$ continues being a local maximum till $\beta<\beta_u'$ where $\beta_u'=\ln (1+\frac{q}{d-1})$; so,  in the interval $(\beta_u,\beta_u')$ both $\nudis$ and $\nuord$ are local maxima of $\Psi$ which causes the metastability phenomena mentioned in the introduction, see \cite{coja2023} for details.

For convenience, let $\Sdis$ and $\Sord$  denote the sets of configurations $S_\nu$ when $\nu=\nudis$ and $\nu=\nuord$, respectively. For a graph $G$,  define  $\Zdis(G)$ and $\Zord(G)$ analogously, and let $\mu_{G,\dis}=\mu_{G}(\cdot \mid\Sdis)$ and $\mu_{G,\ord}=\mu_{G}(\cdot\mid  \Sord)$ denoted the conditional Gibbs distributions. The following lemma from \cite{BIShard} gives some sharp concentration estimates for $Z_{\dis}(\G)$ and $Z_{\ord}(\G)$ (derived using moments and the small subgraph conditioning method) and are the key behind our planting approach.
\begin{lemma}[\mbox{\cite[Theorem 5 \& Lemma 16]{BIShard}}]\label{lem:smallgraph}
Let $d,q\geq 3$ be integers, and $\beta\geq 0$ be real.\footnote{Here, we assume that the constant $\theta>0$, used in the definition of \(\Zdis\) and \(\Zord\),  is a sufficiently small constant depending only on $d,q,\beta$.} Let   $f(n)>0$ be any function with $f(n)=o(1)$. Whp over $\G$, we have:
\[\mbox{ $\Zdis(\G)\geq f(n)\, \Ex[\Zdis(\G)]$ if $\beta\leq \beta_c$} \quad \mbox{and} \quad \mbox{ $\Zord(\G)\geq f(n)\, \Ex[\Zord(\G)]$ if $\beta\geq \beta_c$}.\]
Moreover, $\mu_{\G}\big(\Sdis\big)=1-\emm^{-\Omega(n)}$ when $\beta<\beta_c$, and $\mu_{\G}\big(\tSord\big)=1-\emm^{-\Omega(n)}$ when $\beta>\beta_c$, where $\tSord$ is the union of $\Sord$ with its $q-1$ permutations. For $\beta=\beta_c$, it holds that $\mu_{\G}\big(\Sdis\cup \tSord\big)=1-\emm^{-\Omega(n)}$ and $\mu_{\G}(\Sdis),\mu_{\G}(\Sord)\geq f(n)$.
\end{lemma}

\subsection{Planting  disordered and ordered configurations}
We now introduce the planting distribution which will be the key in obtaining our results. Recall that $\G$ denotes a graph chosen according to the configuration model $\mathcal{G}_{n,d}$; in the expressions below, expectations are taken with respect to the choice of $\G$.    For $\sigma\in[q]^{n}$, define a random graph $\hat\G(\sigma)$ by letting
\begin{align}\label{eq:Gsigma}
\pr\Big[\hat\G(\sigma)=G\Big]&=\frac{\pr\brk{\G=G}\emm^{\beta m_G(\sigma)}}{\Ex[\emm^{\beta m_{\G}(\sigma)}]}.
\end{align}
Furthermore, for $\theta>0$, recalling the disordered and ordered partition functions $\Zdis,\Zord$ from the previous section, we define random configurations $\sdis=\sdis(\theta)\in [q]^n$ and $\sord=\sord(\theta)\in[q]^{n}$ with distributions
\begin{align}\label{eq:disord}	\pr\big[\sdis=\sigma]&=\frac{\vecone\cbc{\sigma\in\Sdis}\Ex[\emm^{\beta m_{\G}(\sigma)}]}{\Ex[\Zdis(\G)]},&
	\pr\big[\sord=\sigma]&=\frac{\vecone\cbc{\sigma\in\Sord}\Ex[\emm^{\beta m_{\G}(\sigma)}]}{\Ex[\Zord(\G)]}.
\end{align}
We refer to $(\hat\G(\sdis),\sdis\big)$ and $(\hat\G(\sord),\sord\big)$ as the disordered and ordered planted distributions respectively.

 The following lemma allows us to relate the planted distributions with the distributions of the (conditional) disordered  $\SIGMA_{\G,\dis}\sim \mu_{\G,\dis}$ and ordered $\SIGMA_{\G,\ord}\sim \mu_{\G,\ord}$ configurations of the random graph $\G$. Recall that $\cG_{n,d}$ denotes the set of all $n$-vertex regular (multi)graphs with degree $d$. 
 \setlist[enumerate,1]{leftmargin=10mm}
 \newcommand{\statelemquiteplanting}{Let $d,q\geq 3$ be integers and $\beta> 0$ be real. Let $f(n),g(n)>0$ be \emph{any} functions with $f(n)=o(1)$. Let $\mathcal{E}\subseteq \cG_{n,d}\times [q]^n$ be an arbitrary event. 
\begin{enumerate}[(1)]
\item $\beta\leq \beta_c$: if $\Pr\big[\big(\hat\G(\sdis),\sdis\big)\in  \cE\big]\leq g(n)$, then, whp over  $\G$, $\Pr\big[(\G,\SIGMA_{\G,\dis}\big)\in\cE\mid \G\big]\leq \frac{g(n)}{f(n)}$.
\item $\beta\geq\beta_c$: if $\Pr\big[\big(\hat\G(\sord),\sord\big)\in  \cE\big]\leq g(n)$, then, whp over  $\G$, $\Pr\big[(\G,\SIGMA_{\G,\ord}\big)\in\cE\mid \G\big]\leq \frac{g(n)}{f(n)}$.
\end{enumerate}}
\begin{lemma}\label{lem:quietplanting}
\statelemquiteplanting
\end{lemma}

The proof is given in Appendix~\ref{sec:prelimsproof}. To utilise Lemma~\ref{lem:quietplanting}, the key observation is that, for the planted models, once we condition on the vertex and edge statistics, the models become uniformly distributed. Formally, for a graph $G\in \cG_{n,d}$ and a configuration $\sigma\in [q]^n$, let $\nu^\sigma$ be vertex colour statistics under $\sigma$, i.e., 
\[\mbox{for a colour $i\in [q]$, $\nu^{\sigma}_i=|\sigma^{-1}(i)|/n$.}\]
Similarly, let $\rho^{G,\sigma}$ be the $[q]\times [q]$ matrix with the following half-edge colour statistics under $\sigma$, i.e., 
\begin{equation}\label{eq:edgestat}
\mbox{for colours $i,j\in [q]$, $\rho^{G,\sigma}_{ij}=\frac{1}{2|E(G)|}\sum_{u,v\in V(G)} \mathbf{1}\big\{\{u,v\}\in E(G),\sigma(u)=i,\sigma(v)=j\big\}$}.
\end{equation}
Recall that $2|E(G)| = d n$.
So, for a colour class $i\in [q]$, out of the $d n \nu_i^\sigma$ half-edges incident to it, there are  $dn\rho^{G,\sigma}_{ij}$ ``going'' to a colour class $j\in [q]$ (note, for $i=j$, there are $\frac{dn}{2}\rho^{G,\sigma}_{ii}$ edges in $G$ that lie inside the colour class $i$). We have that  $\rho^{G,\sigma}$ is symmetric  with $\sum_{i,j}\rho^{G,\sigma}_{ij}=1$. 
Observe further that,  conditioned on $\nu^{\sdis}=\nu$ and $\rho^{\hat\G(\sdis), \sdis}=\rho$ the pair $(\hat\G(\sdis), \sdis)$ is uniformly distributed over the pairs $(G,\sigma)\in \cG_{n,d}\times [q]^n$ with $\nu^{\sigma}=\nu$ and $\rho^{G, \sigma}=\rho$. Similarly for the pair $(\hat\G(\sord), \sord)$ conditioned on  $\nu^{\sord}$ and $\rho^{\hat\G(\sord), \sord}$.

To put this observation into use, we will need more quantitative information on the vertex and edge colour statistics in the planted models. By definition of $\sdis, \sord$, we already  know that $\big\|\nu^{\sdis}-\nudis\big\|_1\leq \theta$ and $\big\|\nu^{\sord}-\nuord\big\|_1\leq \theta$. Analogously, $\rho^{\hat\G(\sdis), \sdis}$ and $\rho^{\hat\G(\sord), \sord}$ are concentrated around the vectors $\rhodis$ and $\rhoord$  where
\begin{equation}\label{eq:rhostatistics}
\mbox{for $i,j\in [q]$}, \quad \rho_{\dis,ij}=\frac{\emm^{\beta\mathbf{1}\{i=j\}}}{\sum_{i',j'\in [q]}\emm^{\beta\mathbf{1}\{i'=j'\}}}, \quad \rho_{\ord,ij}=\frac{\emm^{\beta\mathbf{1}\{i=j\}}(\nu_{\ord,i}\nu_{\ord,j})^{(d-1)/d}}{\sum_{i',j'\in [q]}\emm^{\beta\mathbf{1}\{i'=j'\}}(\nu_{\ord,i'}\nu_{\ord,j'})^{(d-1)/d}}.
\end{equation}
We have the following concentration statement.
\begin{lemma}[\mbox{\cite[Lemmas 3.4 \& 3.5]{coja2023}}]\label{lem:statistics}
Let $d,q\geq 3$ be integers 
 and $\beta> 0$ be real. There exists $c>0$ such that for any  $t=\omega(n^{-1/2})$ it holds that
\begin{enumerate}[(i)]
\item For $\beta< \beta_u'$:  $\Pr\Big[\big\|\nu^{\sdis}-\nudis\big\|_1+\big\|\rho^{\hat\G(\sdis),\sdis}-\rhodis\big\|_1\geq t\Big]\leq \emm^{-c t^2 n}$.
\item For $\beta>\beta_u$:  $\Pr\Big[\big\|\nu^{\sord}-\nuord\big\|_1+\big\|\rho^{\hat\G(\sord),\sord}-\rhoord\big\|_1\geq t \Big]\leq \emm^{-c t^2 n}$.
\end{enumerate}
\end{lemma}

\section{Mixing within the disordered and ordered phases for the RC model}\label{sec:RCmixing}

\subsection{Phases for the random cluster model via Potts}
For a graph $G\in \cG_{n,d}$ and $\sigma\in [q]^n$, let $\F(G,\sigma)$ denote the random subset of edges obtained by keeping each monochromatic edge of $G$ under $\sigma$ with probability $p$. It is a standard fact (see, e.g., \cite{RCMbook}) that, for integer $q\geq 3$, we can obtain a sample $\F\sim \pi_{G}$  by first sampling  $\SIGMA\sim \mu_{G}$ and then keeping each monochromatic edge of $G$ under $\sigma$ with probability $p=1-\emm^{-\beta}$ (the so-called percolation step).

We use this for the random graph $\G$ to get a handle on the random cluster model on $\G$, i.e., $\F(\G,\SIGMA)$   has distribution $\pi_{\G}$. From Lemma~\ref{lem:statistics}, we can figure out in particular how many monochromatic edges we should expect from each phase, so we can translate the phases for the Potts model to the random cluster model.  More precisely, set 
\[\mdis(\beta)=p\frac{d}{2}\sum_{i\in [q]}\rho_{\dis,ii} \mbox{ and } \mord(\beta)=p\frac{d}{2}\sum_{i\in [q]}\rho_{\ord,ii}.\]
With a bit of algebra, we can derive the following expressions for $\mdis$ and $\mord$  from \eqref{eq:rhostatistics}: for $\beta<\beta_u'$, we have $\mdis(\beta)=\frac{d}{2}\frac{\emm^{\beta}-1}{\emm^{\beta}+q-1}$, while, for $\beta>\beta_u$, we have $\mord(\beta)=\frac{d}{2}\frac{(\emm^{\beta}-1)(x^2+\frac{(1-x)^2}{q-1})}{1+(\emm^{\beta}-1)(x^2+\frac{(1-x)^2}{q-1})}$ with $x=\frac{t^{d-1}}{t^{d-1}+q-1}$ and $t>1$ defined  from $\emm^{\beta}-1=\frac{(t-1)(t^{d-1}+q-1)}{t^{d-1}-t}$, as in \eqref{eq:nuord}.

We are now ready to define the ordered/disordered phases for the random-cluster model. Due to monotonicity properties (and in contrast to the phases of the Potts model), we can just use the values of $\mdis$ and $\mord$ at the critical temperature $\beta=\beta_c$; this will be convenient later on.

\begin{definition}[The ordered/disordered phases]\label{def:RCorddis}
Let $q,d\geq 3$ be integers. For a graph $G\in \cG_{n,d}$ and $\rho:=\mord(\beta_c)-\mdis(\beta_c)>0$, define 
\begin{equation*}
\mbox{$\Omega_{G,\dis}=\{F\subseteq E(G): |F|\leq n\mdis(\beta_c)+n\rho/4\}$ and $\Omega_{G,\ord}=\{F\subseteq E(G): |F|\geq n\mord(\beta_c)-n\rho/4\}$.}
\end{equation*}
Let also $\pi_{G,\dis}=\pi_G(\cdot \mid \Omega_{G,\dis})$ and $\pi_{G,\ord}=\pi_G(\cdot \mid \Omega_{G,\ord})$. 
\end{definition}
Using the translation between the Potts and random-cluster models, we have  the following. 
\begin{corollary}\label{cor:smallgraphb}
Let $d,q\geq 3$ be integers and $\beta>0$ be real.  Then, whp over $G\sim \cG_{n,d}$, we have $\pi_{G}\big(\Omega_{G,\dis}\big)=1-\emm^{-\Omega(n)}$ when $\beta<\beta_c$ and $\pi_{G}\big(\Omega_{G,\ord}\big)=1-\emm^{-\Omega(n)}$ when $\beta>\beta_c$. 

For $\beta=\beta_c$, $\pi_{G}\big(\Omega_{G,\dis}\cup \Omega_{G,\ord}\big)=1-\emm^{-\Omega(n)}$ and $\pi_{G}\big(\Omega_{G,\dis}\big),\pi_{G}\big(\Omega_{G,\ord}\big)\geq 1/\ln \ln n$.
\end{corollary}
\begin{proof}[Proof of Corollary~\ref{cor:smallgraphb}]
This follows from Lemma~\ref{lem:smallgraph}, and the correspondence between $\SIGMA\sim \mu_G$ and $\F\sim \pi_G$ stated at the start of the section, using $f(n)=1/\ln \ln n$ for the lower bounds at $\beta=\beta_c$.
\end{proof}

To analyse the ordered/disordered  phases of the random cluster model using the planted distributions, we will need the following intermediate distributions. Namely, for a graph $G$, let $\hat\pi_{G,\dis}$ be the distribution on $F\subseteq E(G)$ obtained by first sampling $\SIGMA_{G,\dis}\sim \mu_{G,\dis}$ and then keeping each monochromatic edge with probability $p=1-\emm^{-\beta}$. Define similarly $\hat\pi_{G,\ord}$.  The proof of the following is given in Appendix~\ref{sec:f4ttvrv}.
\newcommand{\statelemPottstoRC}{Let $d,q\geq 3$ be integers and $\beta>0$ be real. Then,  whp over $G\sim\cG_{n,d}$, for any event $\cA\subseteq 2^{E(G)}$, the following holds.
\begin{enumerate}[(i)]
\item if $\beta\leq \beta_c$, then $\pi_{G,\dis}(\cA)\leq \hat{\pi}_{G,\dis}(\cA)+\emm^{-\Omega(n)}$.  
\item if  $\beta\geq \beta_c$, then $\pi_{G,\ord}(\cA)\leq \hat{\pi}_{G,\ord}(\cA)+\emm^{-\Omega(n)}$.  
\end{enumerate}}
\begin{lemma}\label{lem:PottstoRC}
\statelemPottstoRC
\end{lemma}

\subsection{Mixing and Weak Spatial Mixing (WSM) within a phase}
Treading a path initiated in \cite{SinclairsGheissari2022}, to prove Theorems~\ref{thm:main1} and Theorem~\ref{thm:main2} it suffices to show the following results for the disordered and ordered regimes, respectively.

\newcommand{\statethmRCdis}{Let $d,q\geq 3$ be integers and $\beta\in (0,\beta_c]$. Then, whp over $G\sim\cG_{n,d}$, for every $\eps\geq \emm^{-\Omega(n)}$ the RC dynamics $(X_t)_{t\geq 0}$ initialised from all-out satisfies $\big\|X_T-\pi_{G,\dis}\big\|_{\mathrm{TV}}\leq \eps$ for $T=O(n\log n\log\tfrac{1}{\eps})$.}
\begin{theorem}\label{thm:RCdis}
\statethmRCdis
\end{theorem}

\newcommand{\statethmRCord}{Let $d\geq 3,q\geq (5d)^5$ be integers and $\beta\geq \beta_c$ be a real. Then, whp over $G\sim\cG_{n,d}$, for every $\eps\geq \emm^{-\Omega(n)}$ the RC dynamics $(X_t)_{t\geq 0}$ initialised from all-in satisfies $\big\|X_T-\pi_{G,\ord}\big\|_{\mathrm{TV}}\leq \eps$ for $T=O(n\log n\log\tfrac{1}{\eps})$.}
\begin{theorem}\label{thm:RCord}
\statethmRCord
\end{theorem}
\begin{proof}[Proof of Theorems~\ref{thm:main1} and~\ref{thm:main2}]
We first do the proof of Theorem~\ref{thm:main1}. By Corollary~\ref{cor:smallgraphb}, it holds that $\big\|\pi_G-\pi_{G,\dis}\big\|_{\mathrm{TV}}=\emm^{-\Omega(n)}$ and by Theorem~\ref{thm:RCdis} we can obtain a sample from $\pi_{G,\dis}$ in $T$ steps using random-cluster dynamics. Then, using the translation between the Potts and RC models~\cite{RCMbook}, this gives a sample $\hat\SIGMA$  whose distribution is within TV-distance $\leq \eps$ from $\mu_G$, as wanted. The proof of Theorem~\ref{thm:main2} is analogous using Corollary~\ref{cor:smallgraphb} and Theorem~\ref{thm:RCord}.
\end{proof}

It was shown in \cite{SinclairsGheissari2022} that, for the ferromagnetic Ising model, a sufficient condition for fast convergence to equilibrium within the phase is the property of  \textit{weak spatial mixing (WSM) within a phase}, paired with some mixing-time estimates in local neighbourhoods. The same condition applies for the random-cluster model and, more generally, for monotone models. In the case of the random regular graph, the local neighbourhoods resemble the $d$-regular tree and the required mixing-time results  have been established in \cite{treeRC}. So, the main piece missing in order to obtain Theorems~\ref{thm:RCdis} and~\ref{thm:RCord} is showing WSM within the corresponding phase. Roughly, the notion captures that, when we condition on the phase, the influence of a boundary configuration at distance $\ell$ is small. 

More formally, for a graph $G=(V,E)$, an integer $\ell$ and a vertex $v$, let $B_\ell(v)$ be the graph induced by vertices at distance $\leq \ell$ from $v$, whose vertex and edge sets are going to be denoted by $V(B_\ell(v))$ and $E(B_\ell(v))$, respectively. A boundary condition $\tau$ on $B_\ell(v)$ is a subset of the edges in $E\backslash E(B_\ell(v))$; then $\pi^{\tau}_{G,v,\ell}$ denotes the  RC distribution induced on $B_\ell(v)$ conditioned on $\tau$, i.e., 
for $S\subseteq E(B_\ell(v))$ it holds that $\pi^{\tau}_{G,v,\ell}(S)=\frac{\pi_G(S\cup \tau)}{\sum_{S'\subseteq E(B_\ell(v))} \pi_G(S'\cup \tau)}$. There are two extreme boundary conditions on $B_\ell(v)$ which will be most relevant: (i) the \emph{all-in} or \emph{wired} condition $+$ where all edges in $E\backslash E(B_\ell(v))$ are included in the boundary condition, and (ii) the \emph{all-out} or \emph{free} boundary condition $-$ where none of the edges in  $E\backslash E(B_\ell(v))$ are included. We use $\pi^+_{G,v,\ell}$ and $\pi^-_{G,v,\ell}$ to denote the corresponding conditional RC distributions.

For an edge $e\in E$, let $1_e$ denote the subsets of $E$ that include $e$. Following \cite{RCrandom}, we say that we have WSM within the disordered phase (respectively, ordered) at distance $\ell$ if for each vertex $v$ and each edge $e$ incident to $v$ it holds that $\big|\pi^-_{G,v,\ell}(1_e) - \pi_{G,\dis}(1_e)\big|\leq \frac{1}{100m}$ (respectively, $\big|\pi^+_{G,v,\ell}(1_e) - \pi_{G,\ord}(1_e)\big|\leq \frac{1}{100m}$) with $m=|E|$.

Our main technical results for establishing Theorems~\ref{thm:RCdis} and~\ref{thm:RCord} are the following.
\begin{restatable}{theorem}{wsmdisordered}\label{thm:wsm-disordered}
     Let $d,q\geq 3$ be integers and $\beta\leq \beta_c$ be a positive real. Then, there exists $\delta>0$ such that, whp over $G\sim \cG_{n,d}$, $G=(V,E)$ has WSM within the disordered phase at a distance \(\ell = \lfloor(\frac{1}{2}-\delta)\log_{d-1} n\rfloor\). That is, with $m=|E|$, for every \(v\in V\) and every edge $e$ incident to $v$, it holds that $\big|\pi^-_{G,v,\ell}(1_e) - \pi_{G,\dis}(1_e)\big|\leq \frac{1}{100m}$.
\end{restatable}

\begin{restatable}{theorem}{wsmordered}\label{thm:wsm-ordered}
     Let $d\geq 3, q\geq (5d)^5$ be integers and $\beta\geq \beta_c$ be a real.  Then, there exists $\delta>0$ such that, whp over  $G\sim \cG_{n,d}$, $G$ has WSM within the ordered phase at a distance \(\ell = \lfloor(\frac{1}{2}-\delta)\log_{d-1} n\rfloor\). That is, with $m=|E|$, for every \(v\in V\) and every edge $e$ incident to $v$, it holds that $\big|\pi^+_{G,v,\ell}(1_e) - \pi_{G,\ord}(1_e)\big|\leq \frac{1}{100m}$.
\end{restatable}

Using these, the proofs of Theorems~\ref{thm:RCdis} and~\ref{thm:RCord} can be done using by now standard arguments; we give the details for completeness in Appendix~\ref{sec:proofmain}.

We therefore focus on the WSM proofs. The proof for the disordered regime $\beta\leq \beta_c$ is given in Section~\ref{sec:disordered-regime}, and the proof for the ordered regime $\beta\geq \beta_c$ is given in Section~\ref{sec:wsm-ordered}.

\section{Proof for WSM in the disordered regime}\label{sec:disordered-regime}
To show WSM in the disordered regime, we follow the strategy outlined in Section~\ref{sec:rf4f4f3}. We begin with some relevant definitions. Let $G=(V,E)$ be a graph. Recall, that for a vertex $v$ and $\ell>0$, $B_\ell(v)$ denotes the radius-$\ell$ ball around $v$ in $G$. Also, denote by $S_\ell(v)$ the set of vertices at distance exactly $\ell$ from $v$. For a set of vertices $S$, we let $E(S)$ the set of edges with both endpoints in $S$.

\begin{definition} 
Let $G=(V,E)$ be a graph and $v$ be a vertex. For $K,\ell>0$, we say that a subset $F\subseteq E$ is \emph{$K$-shattered} at distance $\ell$ from $v$ if all but $K$ vertices of $S_\ell(v)$ belong to distinct components of the graph $(V,F\backslash E(B_\ell(v))$.
\end{definition}

\begin{lemma}\label{lem:dispath}
Let $d,q\geq 3$ be integers and $\beta\leq \beta_c$ be a positive real. For any constant $\epsilon>0$, there is $K>0$ so that the following holds whp over $\G$. For every vertex $v\in V$,  define  the event \[\cA_v=\{F\subseteq E(\G)\mid F \mbox{ is $K$-shattered at distance $\ell=\lfloor (\tfrac{1}{2}-2\epsilon)\log_{d-1}n \rfloor$ from $v$}\}.\]  
Then, for the RC disordered phase of $\G$, we have $\pidis(\bigcap_v\cA_v)\geq 1-1/n^3$.
\end{lemma}
\begin{proof}
 For a graph $G\in \cG_{n,d}$ and $\sigma\in [q]^n$, let $\F(G,\sigma)$ denote the random subset of edges obtained by keeping each monochromatic edge of $G$ under $\sigma$ with probability $p$. Fix $\epsilon >0$ and let $K = \lceil 10/\epsilon \rceil$. We will consider the disordered planted graph $(\hat\G(\sdis),\sdis)$ and show   that for an arbitrary vertex $v$ it holds that \begin{equation}\label{eq:key124}
\Pr\Big[\F\big(\hat\G(\sdis),\sdis\big)\notin \cA_v\Big]\leq 2/n^5.
\end{equation}
Assuming this for now, we obtain by Lemma~\ref{lem:quietplanting} and a union bound over the vertices that $\Pr\big[\F\big(\G,\SIGMA_{\G,\dis}\big)\notin\bigcap_{v\in V}\cA_v\mid \G\big]\leq \ln n/n^4 $ as well.
Note that $\F\big(\G,\SIGMA_{\G,\dis})$ has distribution $\hat{\pi}_{{\G},\dis}$, so Lemma~\ref{lem:PottstoRC} yields that  $\pidis(\bigcap_v\cA_v)\geq \hat{\pi}_{\G,\dis}(\bigcap_v\cA_v)-\emm^{-\Omega(n)}\geq 1-1/n^3$.

It remains to prove \eqref{eq:key124}.  By Lemma~\ref{lem:statistics}, we have  
$\big\|\nu^{\sdis}-\nudis\big\|_1+\big\|\rho^{\hat\G(\sdis),\sdis}-\rhodis\big\|_1=O(n^{-1/3})$ w.p. $1-\emm^{-\Omega(n^{1/3})}$. Consider arbitrary $\nu$ and $\rho$ with $\big\|\nu-\nudis\big\|_1+\big\|\rho-\rhodis\big\|_1=O(n^{-1/3})$. Conditioned on $\nu^{\sdis}=\nu$ and $\rho^{\hat\G(\sdis),\sdis}=\rho$, we have that $(\hat\G(\sdis),\sdis)$ is uniformly random among all pairs $(G,\sigma)$ with $\nu^{\sigma}=\nu$ and $\rho^{G,\sigma}=\rho$.

To proceed, fix $\sigma$ with  $\nu^{\sigma}=\nu$ and condition on $\sdis=\sigma$. Condition further that $\rho^{\hat\G(\sdis),\sdis}=\rho$ and on the induced graph on $B_\ell(v)$. For convenience, let $\tilde{\F}$ denote the (random) set of edges outside $B_\ell(v)$ after the percolation step, i.e., the set $\F\big(\hat\G(\sdis),\sdis\big)\backslash E(B_\ell(v))$.  

Let $T:=\lfloor n^{1/2-\epsilon}\rfloor$. For $t=1,\hdots,T$, we explore  the (union of the) components of vertices $u\in S_\ell(v)$ in the graph $(V,\tilde\F)$ by revealing at each step an edge incident to a vertex belonging to any of these (chosen arbitrarily). We denote by $A_t$ the set of ``active vertices'' in these components that still have unmatched half-edges, so that $|A_t|=0$ implies that all the components have been fully explored. Let $A_{\leq t}=\cup_{t'\leq t}A_{t}$ denote the set of all vertices that have been active up to time $t$. The process is described precisely in Figure~\ref{alg:gl}.
\begin{figure}[h]
\begin{mdframed}
\textbf{Exploration of $(V,\tilde\F)$ starting from $S_\ell(v)$} 
\vskip 0.1cm
Let $t=0$, $T=\lfloor n^{1/2-\epsilon}\rfloor$, and $A_0=S_\ell(v)$.\vskip 0.08cm
While $|A_t|\neq 0$ or $t\leq T$: \vskip 0.05cm
\begin{itemize}[leftmargin=0.8cm]
    \item If $|A_t|=0$, pick any unmatched half-edge $(w,r)$ and  run \textsc{Match}$(w,r)$; 

    increase $t$ by 1 and set $A_t=\emptyset$ \vskip 0.15cm
    \item Otherwise, pick $w\in A_{t}$. For $r=1,\hdots,d$:
\begin{description}[before={\renewcommand\makelabel[1]{\bfseries ##1}}]
    \item \ \ \ If  the half-edge $(w,r)$ is not matched:
\begin{description}[before={\renewcommand\makelabel[1]{\bfseries ##1}}]
\item increase $t$ by 1;
\item match $(w,r)$ to another half-edge, let $z=\textsc{Match}$$(w,r)$.
\item if $\sigma(z)=\sigma(w)$, w.p. $p=1-\emm^{-\beta}$ set $A_t\leftarrow A_{t-1}\cup \{z\}$; else set $A_t=A_{t-1}$.
\end{description}
\end{description}
 Remove $w$ and any other  vertices from $A_t$ which don't have unmatched half-edges.
\end{itemize}
\vskip 0.3cm
\textsc{Match}$(w,r)$\\
For $i,j\in [q]$, let $M_{ij,t}$ be the number of currently unmatched half-edges from colour $i$ to $j$.
\begin{itemize}[leftmargin=0.6cm]

\item Let $i=\sigma(w)$. Choose a colour $j\in [q]$ from the distribution  $\{\frac{M_{ij,t}}{\sum_{j'}M_{ij',t}}\}_{j\in [q]}$.
 \item From the set of unmatched half-edges from colour $j$ to $i$, choose one u.a.r., say $(z,r')$.
\item match $(z,r')$ with $(w,r)$,  and return the vertex $z$.
\end{itemize}
\end{mdframed}
\caption{\label{alg:gl} The exploration process of the percolated planted graph used in the proof of Lemma~\ref{lem:dispath}. This is a combination of exposing first an edge  in the planted configuration model, captured by the routine \textsc{Match}$(\cdot,\cdot)$, followed by a percolation step with probability $p=1-\emm^{-\beta}$ whenever the edge is monochromatic.}
\end{figure}

Note that at any time $t$, for $i,j\in [q]$, there is a number  $M_{ij,t}$ of prescribed half-edges with one endpoint of colour $i$ that needs to be matched to a vertex of colour $j$ (note, this goes down when we reveal such an edge connecting colours $i$ and $j$). The revealing of an edge at time $t$ adjacent to an active vertex $w$ with colour $\sigma(w)=i$ has therefore two stages: we first reveal the other endpoint $z$ in the graph $\hat\G(\sigma)$, chosen proportionally to the counts $M_{ij,t}$; if $\sigma(w)=\sigma(z)$,  we further toss a coin with probability $p=1-\emm^{-\beta}$ to determine whether the edge belongs to $\tilde\F$, and add $z$ to the set of active vertices accordingly. 
The key point here is that we only need to reveal the components involving $S_\ell(v)$ in the graph $(V,\tilde\F)$, rather than the components in $\hat\G(\sdis)$ (which are much bigger). Indeed, as we will see shortly the former is dominated by a subcritical branching process and therefore at most $O(\log n)$ vertices are revealed for each of the vertices in $S_\ell(v)$.

To make this precise, for a time $t=1,\hdots, T$, let  $\vecone_t$ denote the indicator of the event that the $t$-th half edge was matched  to a vertex in the set $A_{\leq t}:=\cup_{t'\leq t}A_{t}$ consisting of vertices that have been active at some point. Let also $\vecone_t'$ be the indicator that the edge revealed was monochromatic and that it survived the percolation step. Note that  the number of the remaining unmatched half-edges $M_{ij,t}$ from colour $i$ to $j$ satisfies $0\leq  dn\rho_{ij}-M_{ij,t}\leq 2T\leq 2n^{1/2-\epsilon}$, so $M_{ij,t}= d n\rho_{ij}+o(n)$ for all the relevant $t$. 
Similarly, $|A_{\leq t}|\leq |S_\ell(v)|+T\leq 2n^{1/2-\epsilon}$. So, with $\cE_{ t}$ denoting the first $t$ edges revealed,  we have the crude bounds
\[\Pr(\vecone_{t}\mid \cE_{t-1} )\leq \frac{d |A_{\leq t}|}{ \min _{i}\sum_j M_{ij,t}}\leq \frac{1}{n^{1/2+\epsilon/2}} \quad\mbox{and}\quad \Pr(\vecone_{t}'\mid \cE_{t-1} )\leq p\max_{i\in [q]}\frac{M_{ii,t}}{ \sum_{j} M_{ij,t}}\leq \frac{\emm^{\beta}-1}{\emm^{\beta}+q-1}+o(1).\] 
It follows that the sum $\sum^T_{t=1}\vecone_t$ is dominated above by a sum of $T$ independent coin tosses with probability $\frac{1}{n^{1/2+\epsilon/2}}$, whereas the sum $\sum^T_{t=1}\vecone_t'$ is dominated above by a sum of $T$ independent coin tosses with probability $\frac{\emm^{\beta}-1}{\emm^{\beta}+q-1}+o(1)<\frac{1-2\kappa}{d-1}$ for some small constant $\kappa=\kappa(q,d,\beta)>0$ (using that 
$\emm^\beta - 1 < q/(d-2)$
for $\beta<\beta_u'$). Therefore, recalling that $K=\lceil 10/\epsilon\rceil$ and $T=\lfloor n^{1/2-\epsilon}\rfloor$, 
we have 
\[\Pr\bigg(\sum^T_{t=1}\vecone_t> K\bigg)\leq \binom{T}{K} \frac{1}{n^{K(1/2+\epsilon/2)}}\leq 1/n^5 \quad\mbox{and}\quad \Pr\bigg(\sum^T_{t=1}\vecone_t'\geq \frac{T(1-\kappa)}{d-1}\bigg)=\emm^{-\Omega(n^{1/3})},\]
where the second probability bound follows by standard Chernoff bounds (using that $T= \Omega(n^{1/3})$).

Note that the event $|A_T|>0$ implies that $|A_t|>0$ for all $t=0,1,\hdots, T$ which in turn implies that at least one vertex gets removed from the set of active vertices every (at most) $d-1$ time steps; indeed, when exploring the neighbours of a vertex $w\in A_t$ at a particular time $t$, there are at most $d-1$ half-edges incident to $w$ remaining unmatched since one of the $d$ half-edges of $w$ has been previously matched (to activate $w$). We therefore have the inequality $0\leq |A_{T}|\leq |S_\ell(v)|+\sum^{T}_{t=1} \vecone_t'-\frac{T-(d-1)}{d-1}$.  Since $|S_\ell(v)|\leq d(d-1)^\ell\leq d n^{1/2-2\epsilon}\leq 
\frac{\kappa T}{d-1}$, 
it follows that
\[\Pr(|A_T|>0)\leq \Pr\bigg( |S_\ell(v)|-\frac{T-(d-1)}{d-1}+\sum^{T}_{t=1} \vecone_t'\geq 0\bigg)\leq \Pr\bigg( \sum^T_{t=1}\vecone_t'\geq \frac{T(1-\kappa)}{d-1}\bigg)=\emm^{-\Omega(n^{1/3})}.\]

Combining these, we obtain that with probability $\geq 1-2/n^5$ both of $|A_T|=0$ and $\sum^T_{t=1}\vecone_t< K$ hold. The first implies that by time $T$ all the components of $S_\ell(v)$ in the graph $(V,\tilde\F)$ have been explored, and the second implies that all but (at most) $K$ of these vertices belong to distinct components. This finishes the proof of \eqref{eq:key124} and hence completes the proof of the lemma.
\end{proof}

We next use an argument of Blanca and Gheissari \cite{BG} which shows a correlation decay bound for $K$-shattered configurations in a treelike graph in terms of the existence of two root-to-leaf paths. The argument in \cite{BG} is technically proved for $\beta<\beta_u$ using some uniqueness estimates, but it can be extended it to the regime $\beta\leq \beta_c$ by observing that in a treelike ball, a $K$-shattered configuration induces a distribution that is comparable to percolation on the tree with parameter $\hat p$ (this coincides with the cut-edge update probability in the definition of the random cluster dynamics). A key feature of the correlation decay bound in \cite{BG} is the existence of two paths which gives the $\hat p^{2\ell}$ term in the lemma below, and which ultimately allows us to make the probability less than $\frac{1}{100m}$ that is required for WSM, while keeping the distance $\ell$ less than $\frac{1}{2}\log_{d-1}n$ that is the threshold for having treelike neighbourhoods. We give the proof of the following lemma in 
Appendix~\ref{sec:appendix_proof4.3}.

\newcommand{\statelemdecay}{    Let \(d,q\geq 3\) be integers, and \(\beta \leq \beta_c\), \(\delta\in(0,1/2)\) positive reals.

    There exists a constant \(C = C(d,q,\beta,\delta) > 0\) such that whp over \(G\sim\cG_{n,d}\) the following holds. For all \(v\in V(G)\) and an edge \(e\) incident to it, and \(\ell := \lfloor(\tfrac12 - \delta)\log_{d-1}n\rfloor\),
    \[
    \big|\pi_{G,v,\ell}^-(1_e)-\pi_{G,\dis}(1_e)\big| \leq C \hat p^{2\ell} + \frac{1}{n^3} \mbox{ where $\hat p := \frac{\emm^\beta - 1}{q + \emm^\beta - 1}$.}
    \]}
\begin{lemma}\label{lem:2l-decay-bound-for-twopath}
\statelemdecay
\end{lemma}

We can now give the proof of Theorem~\ref{thm:wsm-disordered}.

\wsmdisordered*
\begin{proof}    
    Consider \(\beta\leq \beta_c\). Since $\beta_c<\beta_u'=\ln(1+\frac{q}{d-2})$, we have that $\hat p(\beta) < \hat p(\beta_u')=\frac{1}{d-1}$. Hence there is 
    an $\epsilon(q,d,\beta)>0$ such that \(\hat p(\beta) = \frac{1}{(d-1)^{1+\epsilon}}\). Let \(\delta = \frac{\epsilon}{4(1+\epsilon)}\) and \(\ell = \lfloor(\frac{1}{2}-\delta)\log_{d-1}n\rfloor\).
    
    Lemma~\ref{lem:2l-decay-bound-for-twopath} gives that, whp over \(G\sim\cG_{n,d}\), for any edge \(e\) and vertex \(v\) incident to it, it holds that \( \big|\pi_{G,v,\ell}^-(1_e)-\pi_{G,\dis}(1_e)\big|\leq C \hat p^{2l} + \frac{1}{n^3}\) . By the choice of \(\delta,\epsilon\) and \(\ell\), we get that $\hat p^{2l}\leq  n^{-(1+\epsilon/2)}/(\hat p^2)$, so for large enough $n$
   we have $C{\hat p}^{-2} n^{-(1+\epsilon/2)} + n^{-3} \leq \frac{1}{100m}$, concluding the proof of WSM within the disordered phase.
\end{proof}

\section{Proof for WSM within the ordered regime}\label{sec:orderedproof}

To show WSM within the ordered phase for the random-cluster model on $G\sim \cG_{n,d}$, we follow the strategy outlined in Section~\ref{sec:rf4f4f3}. The core of the argument will be to show that there are no paths of length $\ell\geq 1/2\log_{d-1}n$ that avoid the largest component in a typical configuration from $\pi_{G,\ord}$. We will do this by considering percolation on  the planted random graph model $(\hat\G(\sord),\sord)$  for $\mu_{G,\ord}$. It will thus be relevant to  consider random graph percolation on (almost) $d$-regular random graphs where we have removed a path of length $\ell$.

For a graph $G$, let $\cC_1(G)$ denote the largest component of $G$ (breaking ties arbitrarily), and more generally, $\cC_i(G)$ denote the $i$-th largest component of $G$. It is a standard fact \cite{perc1,perc2,coja2023} that the size of the largest component in a percolated random $d$-regular graph is controlled by the extinction probability $\varphi$ of a Galton-Watson process with distribution $\mathrm{Bin}(d-1,p)$ on the tree. For $p\in (\tfrac{1}{d-1},1)$, define $\varphi=\varphi(p), \chi=\chi(p)$ and $\hat\varphi=\hat\varphi(p)$ to be the unique solutions on the interval $(0,1)$ to the following
\begin{equation}\label{eq:phi}
\varphi=(1-\varphi+p\varphi)^{d-1},\    \chi:=1-(p \varphi+1-p)^d, \ \hat\varphi:=(1-\varphi+p\varphi)^{d-2}.
\end{equation}
In particular, the size of the largest component is roughly equal to $\chi n+o(n)$; see \cite{perc1,perc2} for more details. We also remark that the quantity $\hat\varphi$ is the extinction probability of a slightly modified GW-tree where every vertex has degree $d$ apart from the root which has degree $d-2$; this will be relevant later for our path considerations. 

We show a similar result for the size of the giant component for slightly non-regular degree sequences $\vec d$, where we allow a small set~$S$ of vertices to have degree less than $d$; in our WSM setting, we will further be interested in the probability that a particular subset of the vertices $S'$ does not belong to the giant. A variant of percolation we will use is the \emph{exact edge model} for random graphs with a given degree sequence $\vec d = (d_1,\dots,d_n)$ and edge count $m\leq \tfrac12\sum d_i$. Let $\tilde\G_{n,\mathbf{d},m}$ be the random graph model where each $i\in [n]$ is associated to $d_i$ half-edges; then we choose a subset of $2m$ half-edges and a matching of these uniformly at random; the final (multi)graph is obtained by projecting the edges on the vertex set $[n]$. 
\newcommand{\statelemgiantperc}{Let $d\geq 3$ be an integer and $p\in (0,1)$. Let $\delta,\eps>0$ be arbitrarily small constants.

Suppose that $S\subseteq [n]$ satisfies $|S|=O(n^{1/6})$ and  $\mathbf{d}=(d_1,\hdots, d_n)$ is a degree sequence such that $d_i=d$ for $i\notin S$ and $d_i\in [0,d]$ for $i\in S$. Then, there exists a constant \(M(d, p) > 0\), such that for $\tilde\G_{n,\mathbf{d},m}$ with $m=\tfrac{1}{2}dpn+o(n)$, the following hold:
\begin{enumerate}
    \item If $p<\tfrac{1}{d-1}$, $\Pr\big[\cC_1(\tilde\G_{n,\vec d,m})\geq Mn^{1/2}]\leq \emm^{-\Omega(n)}$.
    \item If $p>\frac{1}{d-1}$, $\Pr\big[\cC_1(\tilde\G_{n,\vec d,m})\leq(\chi- \delta)n]\leq \emm^{-\Omega(n)}$ and $\Pr\big[\cC_2(\tilde\G_{n,\vec d,m})\geq Mn^{2/3}]\leq \emm^{-\Omega(n)}$. Moreover, with $\varphi=\varphi(p)$ as in \eqref{eq:phi}, for any subset $S'\subseteq S$,  
\[\Pr\big[S'\cap  \cC_1(\tilde\G_{n,\vec d,m})=\emptyset\big]\leq(1+\eps)^{|S|}\prod_{i\in S'}  (p \varphi+1-p)^{d_i}.\]
\end{enumerate}}
\begin{lemma}\label{lem:giantperc}
\statelemgiantperc
\end{lemma}
Intuitively, the bound in the second item follows by revealing the $R$-neighbourhoods of vertices in $S'$ for a suitably large constant $R$ (these are whp trees and disjoint from each other); then, for a vertex $i\in S'$, the factor $(p \varphi+1-p)^{d_i}$ captures the probability that the percolation dies out within the neighourhood of the vertex $i$. The full proof of Lemma~\ref{lem:giantperc}  is given in Appendix~\ref{sec:appendix_proof5.1}.

Next, we consider the density of the edges inside the color classes. The following lemma says that in the ordered regime one colour class is supercritical, while the others subcritical.
Recall that \(\nu_\ord\) and \(\rho_\ord\) are the vertex and edge statistics for the ordered phase respectively, defined in~\eqref{eq:nuord} and~\eqref{eq:rhostatistics}.

\begin{lemma}[{\cite[Lemma 5.6]{coja2023}}]\label{lem:percord}
Let $q,d\geq 3$ be integers. For $\beta>\beta_u$ and $p=1-\emm^{-\beta}$,  the ordered vectors $\nu_\ord$ and $\rho_\ord$ satisfy $p\frac{\rho_{\ord,ii}}{\nu_{\ord,i}}>\frac{1}{d-1}$ for $i=1$,  and $p\frac{\rho_{\ord,ii}}{\nu_{\ord,i}}<\frac{1}{d-1}$ for $i\in \{2,\hdots,q\}$.
\end{lemma}

Let $G=(V,E)$ be a graph. For a subset of edges $F\subseteq E$ and a vertex $v$, we let $F_v$ be the edges in $F$ incident to $v$. We will also consider  the largest component $\cC$ of the graph $(V,F)$, i.e. the component with the largest number of vertices (breaking ties arbitrarily). We say that $\cC$ avoids a path $P$ if there is no vertex $v\in \cC$ that belongs to $P$. The next lemma is the key technical step in our argument and will be used to show that, in the RC ordered phase, the largest component (which has size $\Omega(n)$ due to the supercriticality of colour 1, cf. Lemma~\ref{lem:percord})  is ``very close'' to every vertex $v$.

\begin{lemma}\label{lem:ordpath}
Let $d,q\geq 3$ be integers, $\beta\geq \beta_c$ be real and $p=1-\emm^{-\beta}$. For any constant $\eps>0$, the following holds for any integer $\ell$  satisfying $\ln \ln n\leq \ell\leq n^{1/6}$, whp over $G\sim\cG_{n,d}$.

For a vertex $v\in V(G)$, consider  the event \[\cA_{v}=\Big\{F\subseteq E(G) \,\Big|\,\begin{array}{c}\mbox{there exists a path of length $\ell$ from $v$ in $G$ which}
\\ \mbox{avoids the largest component of $(V\backslash \{v\},F\backslash F_v)$}\end{array}\Big\}.\]  
  Then, for the RC ordered phase of $G$,  
  $\pi_{G,\ord}\big(\bigcup_{v\in V(G)}\cA_{v}\big)\leq n d(d-1)^{\ell-1}(A+\eps)^\ell$ where $A=\max_{i\in [q]}\big\{ 1-(1-{ \sqrt{\hat\varphi_1})}\frac{\rho_{\ord,i1}}{\nu_{\ord,i}}\big\}$ with $\rho_{\ord},\nu_{\ord}$ as in \eqref{eq:nuord} and $\hat\varphi_1=\hat\varphi(p_1)$ as in \eqref{eq:phi} for $p_1=p\frac{\rho_{\ord,11}}{\nu_{\ord,1}}$.
\end{lemma}
To parse the bound on $\pi_{G,\ord}\big(\bigcup_{v\in V(G)}\cA_{v}\big)$ note that the factor $n d(d-1)^{\ell-1}$ is just an upper bound on the total number of paths of length $\ell$. The factor $A^{\ell}$ controls the probability that a particular path $P$ avoids the largest component. The term $\hat\varphi_1$ in $A$ is the probability that the Galton-Watson process dies out from a vertex $v$ of the path that belongs to the supercritical color class 1 (note that $v$ has ``remaining'' degree $d-2$  other than its neighbours on the path, and that the relevant percolation parameter is given by $p_1$). The term $\frac{\rho_{\ord,i1}}{\nu_{\ord,i}}$ accounts for the frequency that we encounter vertices in the supercritical color class 1 along the path $P$, which follows from the definition of the planting model. The exact quantitative dependence of $A$ to $\hat\varphi_1$ and $\frac{\rho_{\ord,i1}}{\nu_{\ord,i}}$ follows by a more technical calculation involving eigenvalues. Later, in Lemma~\ref{lem:ordered-less-than-(d-1)^-3}, we will show that $A$ can be made less than $1/d^5$ by taking $q$ sufficiently large.

\begin{proof}[Proof of Lemma~\ref{lem:ordpath} (Sketch)]
From Lemmas~\ref{lem:quietplanting} and~\ref{lem:PottstoRC}, it suffices to show for the ordered planted graph $(\hat\G(\sord),\sord)$   that (for all $n$ sufficiently large)
\begin{equation}\label{eq:key124b}
\Pr\Big[\F\big(\hat\G(\sord),\sord\big)\in \bigcup_{v\in V(G)}\cA_{v}\Big]\leq nd\big(d-1)^{\ell-1}(A+\tfrac{\eps}{2}\big)^{\ell}.
\end{equation}
From this, we obtain by Lemma~\ref{lem:quietplanting}  that $\Pr\big[\F\big(\G,\SIGMA_{\G,\ord}\big)\in\bigcup_{v\in V}\cA_v\mid \G\big]\leq n\ln (n) d\big(d-1)^{\ell-1}(A+\tfrac{\eps}{2}\big)^{\ell}$; since  $\F\big(\G,\SIGMA_{\G,\ord})$ has distribution $\hat{\pi}_{{\G},\ord}$,  Lemma~\ref{lem:PottstoRC} yields that, whp over $\G$, it holds that $\piord(\bigcup_v\cA_v)\leq \hat{\pi}_{{\G},\ord} (\bigcup_v\cA_v)+\emm^{-\Omega(n)}\leq n d(d-1)^{\ell-1}(A+\eps)^\ell$.

By Lemma~\ref{lem:statistics}, we have $\big\|\nu^{\sord}-\nuord\big\|_1+\big\|\rho^{\hat\G(\sord),\sord}-\rhoord\big\|_1=O(n^{-1/3})$ w.p. $1-\emm^{-\Omega(n^{1/3})}$.  Condition on $\nu^{\sord}=\nu$ and $\rho^{\hat\G(\sord),\sord}=\rho$ where $\nu$ and $\rho$ satisfy $\|\nu-\nuord\|_1+\|\rho-\rhoord\|_1=O(n^{-1/3})$, and note that the pair $(\hat\G(\sord),\sord)$ is uniformly random from the pairs having these statistics. 

Let $P$ be an  arbitrary sequence of vertices $(v_0,v_1,\hdots,v_\ell)$, together with a $q$-colouring $\tau=(\tau_0,\hdots, \tau_\ell)$. 
 Let also $\xi$ be a set of paired half-edges that turn $v_0,\hdots,v_\ell$ into a path in that order. Let \(\E\) be the set of random edges in the graph before percolation. Denote by 
 \[\cE_{P,\tau,\xi}=\{\xi\subseteq \E,\ \hat\SIGMA_\ord(v_0)=\tau_0,\hdots,\  \hat\SIGMA_\ord(v_\ell)=\tau_\ell\}\] the event that $v_0,v_1,\hdots,v_\ell$ form a path in $\hat\G$ (in this order) using the  half-edges prescribed by $\xi$ and that they belong to the colour classes that $\tau$ prescribes. 
In Appendix~\ref{sec:E1E1}, we give an exact expression for $\Pr(\cE_{P,\tau,\xi})$ in terms of $\nu$ and $\rho$ which yields that
    \begin{equation}\label{eq:Ptauxi}
 \Pr(\cE_{P,\tau,\xi})\sim (dn)^{-\ell}\nu_{\ord,\tau_0}\prod^{\ell-1}_{k=0}\frac{\rho_{\ord, \tau_k\tau_{k+1}}}{\nu_{\ord,\tau_k}},
 \end{equation}
 where \(f(n)\sim g(n)\) denotes \(\frac{f(n)}{g(n)}\rightarrow 1\) as \(n\rightarrow \infty\).

Fix now an arbitrary vertex $v$ and let $P$ be an  arbitrary sequence $(v_0,v_1,\hdots,v_\ell)$ with $v_0=v$, together with a $q$-colouring $\tau=(\tau_0,\hdots, \tau_\ell)$ and pairing $\xi$ (as above). Condition on the event $\cE_{P,\tau,\xi}$. Let $\F$ be the edge set obtained by keeping each monochromatic edge with probability $p$, and  $\F_v$ be the set of these edges that are incident to $v$. Then, for the graph $\hat{\G}_v(\F):=(\V\backslash \{v\},\F\backslash \F_v)$, we show in Appendix~\ref{sec:E1E1} that  color 1 is supercritical and the others are subcritical (using Lemma~\ref{lem:percord}); by invoking then Lemma~\ref{lem:giantperc} to the supercritical color class 1, we further obtain that
\begin{equation}\label{eq:step13223}
\Pr\Big[\begin{array}{c}\mbox{the largest component of $\hat{\G}_v(\F)$}\\\mbox{avoids the $\tau$-coloured path $P$}\end{array}\,\Big|\, \cE_{P,\tau,\xi}\Big]\leq (1+\tfrac{\eps}{3})^{\ell} \hat\varphi^{\ell_1(\tau)}_1,
\end{equation}
where $\ell_1(\tau)$ is the number of vertices with colour 1 under $\tau$ and $\hat \varphi_1:=(p_1 \varphi_1+1-p_1)^{d-2}$ where $p_1=p\frac{\rho_{\ord,11}}{\nu_{\ord,1}}$ and $\varphi_1=\varphi(p_1)$ (as in \eqref{eq:phi}). Combining \eqref{eq:Ptauxi} and \eqref{eq:step13223} via a union bound (note that there are $\sim n^\ell$ sequences $P$ with $\ell+1$ vertices starting from $v$ and $d^2(d(d-1))^{\ell-1}$ ways to select half-edges $\xi$ to turn each of these into a path) yields that
\begin{equation}\label{eq:step13}
\Pr\Big[\begin{array}{c}\mbox{the largest component of $\hat{\G}_v(\F)$ avoids}\\\mbox{some $\tau$-coloured path from~$v$ in~$G$}\end{array}\Big]\leq (1+\tfrac{\eps}{3})^{\ell}d(d-1)^{\ell-1}\nu_{\ord,\tau_0}\prod^{\ell-1}_{k=0}\frac{\rho_{\ord, \tau_k\tau_{k+1}}}{\nu_{\ord,\tau_k}}\hat\varphi_1^{\mathbf{1}\{\tau_{k}=1\}}.
\end{equation}
At this stage, it remains to take a union bound over all possible colourings $\tau$ on $\ell+1$ vertices, i.e., sum the right-hand-side of \eqref{eq:step13} over all $\tau$. The resulting expression can be written more succinctly  in the form $(1+\tfrac\eps 3)d(d-1)^{l-1}\cdot\nu_\ord^T M^\ell 1$ where $ M$ is the $q\times q$ matrix whose $(i,j)$-entry is given by $\frac{\rho_{\ord,ij}}{\hat\nu_{\ord,i}}$ and $\hat\nu_\ord$ is the vector whose $i$-th entry is equal to $\nu_{\ord,1}/\hat\varphi_1$ if $i=1$ and $\nu_{\ord,i}$ if $i\in \{2,\hdots,q\}$. Note that $ M=D^{-1/2}\hat M D^{1/2}$ where $D$ is the $q\times q$ diagonal matrix with $\hat \nu_\ord$ on the diagonal and $\hat M$  is the symmetric matrix $\frac{\rho_{\ord,ij}}{(\hat\nu_{\ord,i}\hat\nu_{\ord,j})^{1/2}}$. In particular, we have that $\hat M$ and $M$ have the same eigenvalues, which are all real. Let  $\sqrt{\hat \nu_\ord}$ denote the vector whose entries are the square roots of the entries of $\hat \nu_\ord$. Then, we can write  
\begin{equation*}
\mbox{$\hat \nu_{\ord}^T  M^\ell 1=\sqrt{\hat \nu_{\ord}}^T \hat M^{\ell} \sqrt{\hat \nu_{\ord}}\leq \tfrac{1}{\hat \varphi_1}\sqrt{\nu_{\ord}}^T \hat M^{\ell} \sqrt{\nu_{\ord}}$}.
\end{equation*}
Now, we will show that entry-wise it holds that $\hat M \sqrt{\nu_{\ord}}\leq A \sqrt{\nu_{\ord}}$ where $A:= \max_{i\in [q]}\{1-\frac{(1-\sqrt{\hat\varphi_1})\rho_{\ord,i1}}{\nu_{\ord,i}}\}$ is as in the lemma statement. Note first that  $(\hat M \sqrt{\nu_{\ord}})_i=\sum_{j}\frac{\rho_{\ord,ij}}{(\hat\nu_{\ord,i}\hat\nu_{\ord,j})^{1/2}}(\nu_{\ord,j})^{1/2}$.  Since $\hat \nu_{\ord,j}=\nu_{\ord,j}$ for $j\neq 1$ and   $\sum_{j}\rho_{\ord,ij}=\nu_{\ord,i}$ for $i\in [q]$, we have that
\[( \hat M \sqrt{\nu_{\ord}})_1=\frac{\hat \varphi_1\, \rho_{\ord,11}} {(\nu_{\ord,1})^{1/2}}+\sum_{j\neq 1}\frac{\sqrt{\hat\varphi_1} \rho_{\ord,1j}}{(\nu_{\ord,1})^{1/2}}=\frac{(\hat\varphi_1-\sqrt{\hat\varphi_1})\rho_{\ord,11}}{(\nu_{\ord,1})^{1/2}}+\sqrt{\hat\varphi_1}(\nu_{\ord,1})^{1/2}\leq \sqrt{\hat\varphi_1} A(\nu_{\ord,1})^{1/2}.\]
For $i\neq 1$, we have 
\[( \hat M \sqrt{\nu_{\ord}})_i=\frac{\sqrt{\hat \varphi_1} \rho_{\ord,i1}}{(\nu_{\ord,i})^{1/2}}+\sum_{j\neq 1}\frac{ \rho_{\ord,ij}}{(\nu_{\ord,i})^{1/2}}=(\nu_{\ord,i})^{1/2}-\frac{(1-\sqrt{\hat\varphi_1})\rho_{\ord,i1}}{(\nu_{\ord,i})^{1/2}}\leq A(\nu_{\ord,i})^{1/2}.\]
This finishes the proof of $\hat M \sqrt{\nu_{\ord}}\leq A \sqrt{\nu_{\ord}}$, and therefore we can bound the left-hand-side of \eqref{eq:step13} with $d(d-1)^{\ell-1}(A+\tfrac{\eps}{2})^\ell$. Noting that the event in \eqref{eq:step13} is just $\Pr\Big[\F\big(\hat\G(\sord),\sord\big)\in \bigcup_{v\in V(G)}\cA_{v}\Big]$, 
\eqref{eq:key124b} follows by a union bound over $v$.
\end{proof}

We next bound the quantity $A$ appearing in Lemma~\ref{lem:ordpath} by a suitably small constant so that we can take a union bound over paths; it is for this step that we require $q\geq (5d)^5$. The proof is in  Appendix~\ref{sec:proof-lemma-18}.

\newcommand{\stateOrderedUpperBound}{Let $d\geq 3$, $q\geq (5d)^5$ and $\beta\geq \beta_c$. Then, for all $i\in [q]$, it holds that ${(1-\sqrt{\hat\varphi_1})}\frac{\rho_{\ord,i1}}{\nu_{\ord,i}}>1-\frac{1}{(d-1)^5}$.}
\begin{lemma}\label{lem:ordered-less-than-(d-1)^-3}
\stateOrderedUpperBound
\end{lemma}

\subsection{WSM within the ordered phase}\label{sec:wsm-ordered}

Next we prove WSM within the ordered phase. We first define a \textit{wired boundary} and show how it implies the WSM within the ordered phase, and then prove that it occurs with a good probability.

\begin{definition}
Let \(G = (V,E)\) be a graph, \(v\in V\), \(\ell\geq 1\) and let \(F\subseteq E\) be a RC configuration. Consider the $\ell$-neighbourhood $B_\ell(v)$ of $v$ in $G$. We say that \(S\subseteq V(B_\ell(v))\setminus\{v\}\) is a \emph{wired boundary around \(v\) in \(B_\ell(v)\)}, if

    \begin{itemize}
        \item \(S\) is a cut set separating \(v\) from the outside of the ball. That is, \(\cC_v\), the component of \(v\) in \(G\setminus S\), is disjoint from \(G\setminus B_\ell(v)\); and

        \item in the graph \((V\setminus V(\cC_{v}), \{e \in F : e\cap V(\cC_v) = \emptyset\})\), all vertices of \(S\) are in the same component.
    \end{itemize}

\end{definition}

\newcommand{\stateTVastmostWB}{Let \(d, q\geq 3\) be integers, \(\beta\geq \beta_c\) a real. Then whp over \(G\sim\cG_{n,d}\) the following holds.
    For any \(v\in V\), any edge \(e\) incident to \(v\), and any \(\ell \leq \frac{1}{2}\log_{d-1} n\),
    \[
    \big|\pi_{G,v,\ell}^+(1_e) - \pi_{G,\ord}(1_e)\big|\leq \pi_{G,\ord}\big(\mbox{\emph{no wired boundary around $v$ in  $B_\ell(v)$}}\big) + \emm^{-\Omega(n)}.
    \]}
\begin{lemma}\label{lem:wsm-implied-from-wired-boundary}
    \stateTVastmostWB
\end{lemma}
 To prove the lemma, we compare with \(\pi_G\). When \(\beta > \beta_c\), we obtain, from Corollary~\ref{cor:smallgraphb}, that $\left\|\pi_{G,\ord}-\pi_G\right\|_{\mathrm{TV}}=\emm^{-\Omega(n)}$.
Then we can bound \(|\pi_{G,v,\ell}^+(1_e) - \pi_G(1_e)|\) by the probability (under \(\pi_G\)) that there is no wired boundary around \(v\) in \(B_\ell(v)\): informally, this is because probability of \(1_e\) conditional on the wired boundary is the same under both \(\pi_G\) and \(\pi_{G,v,\ell}^+\), and due to the monotonicity of the RC distribution, the probability of the wired boundary under \(\pi_{G,v,\ell}^+\) is at least that under \(\pi_G\). For more details as well as the case when $\beta=\beta_c$ see the full proof of the Lemma is in the Appendix~\ref{sec:wired-boundary-proofs}.

By Lemma~\ref{lem:ordpath} all simple length-\(l\) paths from \(v\) intersect the largest component of \(G_{v,F}\), so we can construct a wired boundary $S$ by taking, for each length-\(\ell\) path from \(v\) in \(G\), the first vertex on the path belonging to the largest component of \(G_{v,F}\). We thus show the following lemma in Appendix~\ref{sec:wired-boundary-proofs}. 

\newcommand{\stateWBfromgiantpaths}{Let \(G = (V,E)\) be a graph,  let \(v\in V\) be an arbitrary vertex, and let \(\ell\geq 1\) an integer. For any \(F\subseteq E\), 
let 
$F_v$ be the set of edges in~$F$ that are incident to~$v$ and let
\(G_{v,F} := (V\setminus\{v\},F\setminus F_v)\). 
Suppose that all simple length-$\ell$ paths  from $v$ in $G$ intersect  the largest component of $G_{v,F}$.
    Then \(F\) has a wired boundary around \(v\) in \(B_\ell(v)\).}
\begin{lemma}\label{lem:no-long-path-outside-giant=>wired}
\stateWBfromgiantpaths
\end{lemma}

From these two lemmas we can now obtain WSM within the ordered phase.
\wsmordered*
\begin{proof}
By Lemma~\ref{lem:ordered-less-than-(d-1)^-3},  we have $A := \{\max_i (1 - (1-\sqrt{\hat\varphi_1}))\frac{\rho_{o,i1}}{\nu_{o,i}}\} < (d-1)^{-5}$.  Let $\eps > 0$ be a small constant (depending on~$d$) such that $A+\eps \leq (d-1)^{-(5+\eps)}$. Let $\delta := \frac{\eps}{4(4+\eps)}$.  By Lemma~\ref{lem:wsm-implied-from-wired-boundary}, 
 $
    \big|\pi_{G,v,\ell}^+(1_e) - \pi_{G,\ord}(1_e)\big|$
is at most   
 \[ \emm^{-\Omega(n)} + 1-  \pi_{G,\ord}(\exists  \text{ a wired boundary around }v\text{ in } B_\ell(v))  .\]
By Lemma~\ref{lem:no-long-path-outside-giant=>wired}, this is at most
 \[ \emm^{-\Omega(n)} + 1-  \pi_{G,\ord}(  \text{ 
 all simple length-$\ell$ paths from $v$ in $G$ intersect the largest component of $G_{v,F}$ }) .\]
By Lemma~\ref{lem:ordpath}  
this is at most $\emm^{-\Omega(n)} 
+ n d (d-1)^{\ell-1}(A+\eps/2)^\ell
$. So whp over $G\sim\cG_{n,d}$ it holds for every vertex $v$ and edge \(e\) incident to it that
    \[
    \big|\pi_{G,\ord}(1_e)-\pi_{G,v,\ell}^+(1_e)\big|\leq \emm^{-\Omega(n)} + \frac{nd}{d-1} (d-1)^{-(4+\eps)\ell} 
    \leq 
    \emm^{-\Omega(n)} + \frac{1}{n^{1+\eps/4}}
    \]
which is at most $ \frac{1}{100m}$ for all $n$ large enough.
\end{proof}

\printbibliography

@article{BIShard,
author = {Galanis, Andreas and \v{S}tefankovi\v{c}, Daniel and Vigoda, Eric and Yang, Linji},
title = {Ferromagnetic {P}otts Model: Refined \#{BIS}-hardness and related results},
journal = {SIAM Journal on Computing},
volume = {45},
number = {6},
pages = {2004-2065},
year = {2016}
}

@inproceedings{SinclairsGheissari2022,
  author    = {Reza Gheissari and
               Alistair Sinclair},
  title     = {Low-temperature {I}sing dynamics with random initializations},
  booktitle = {54th Annual {ACM} {SIGACT} Symposium on Theory of Computing, \emph{(STOC '22)}},
  pages     = {1445--1458},
  year      = {2022}
}

@article{bencs2022random,
  title={Random cluster model on regular graphs},
  author={Bencs, Ferenc and Borb{\'e}nyi, M{\'a}rton and Csikv{\'a}ri, P{\'e}ter},
  journal={Communications in Mathematical Physics},
  pages={1--46},
  year={2022},
}

@inproceedings{BIS1,
  author       = {Matthew Jenssen and
                  Peter Keevash and
                  Will Perkins},
  title        = {Algorithms for {\#}BIS-hard problems on expander graphs},
  booktitle    = {Proceedings of the Thirtieth Annual {ACM-SIAM} Symposium on Discrete
                  Algorithms, {SODA} 2019},
  pages        = {2235--2247},
  year         = {2019}
}

@article{gheissari2025rapid,
  title={Rapid phase ordering for {I}sing and Potts dynamics on random regular graphs},
  author={Gheissari, Reza and Sly, Allan and Sohn, Youngtak},
  journal={arXiv preprint arXiv:2505.15783},
  year={2025}
}

@inproceedings{BIS2,
  author       = {Zongchen Chen and
                  Andreas Galanis and
                  Leslie Ann Goldberg and
                  Will Perkins and
                  James Stewart and
                  Eric Vigoda},
  title        = {Fast algorithms at low remperatures via {M}arkov chains},
  booktitle    = {Approximation, Randomization, and Combinatorial Optimization. Algorithms
                  and Techniques, {APPROX/RANDOM} 2019},
  volume       = {145},
  pages        = {41:1--41:14},
  year         = {2019}
}

@article{perc1,
	author = {Noga Alon and Itai Benjamini and Alan Stacey},
	title = {Percolation on finite graphs and isoperimetric inequalities},
	volume = {32},
	journal = {The Annals of Probability},
	number = {3},
	publisher = {Institute of Mathematical Statistics},
	pages = {1727 -- 1745},
	year = {2004}
}

@article{perc2,
	title={Asymptotics in percolation on high-girth Expanders},
	author={Krivelevich, Michael and Lubetzky, Eyal and Sudakov, Benny},
	journal={Random Structures and Algorithms},
	volume={56},
	number={4},
	pages={927--947},
	year={2020}
}

@InProceedings{treeRC,
  author =	{Blanca, Antonio and Chen, Zongchen and \v{S}tefankovi\v{c}, Daniel and Vigoda, Eric},
  title =	{The {S}wendsen-{W}ang dynamics on trees},
  booktitle =	{Approximation, Randomization, and Combinatorial Optimization. Algorithms and Techniques \emph{(APPROX/RANDOM 2021)}},
  pages =	{43:1--43:15},
  year =	{2021},
  volume =	{207}
}

@inproceedings{RClattice,
author = {Reza Gheissari and Alistair Sinclair},
title = {Spatial mixing and the random-cluster dynamics on lattices},
booktitle = {Proceedings of the 2023 Annual ACM-SIAM Symposium on Discrete Algorithms, \emph{(SODA '23)
}},
pages = {4606--4621},
year      = {2023}
}

@article{BlaGhe,
  author       = {Antonio Blanca and
                  Reza Gheissari},
  title        = {On the tractability of sampling from the {P}otts model at low temperatures
                  via {S}wendsen-{W}ang dynamics},
  journal      = {CoRR},
  volume       = {abs/2304.03182},
  year         = {2023}
}

@article{RCrandom, title={Sampling from the random cluster model on random regular graphs at all temperatures via Glauber dynamics}, volume={34}, DOI={10.1017/S0963548324000403}, number={3}, journal={Combinatorics, Probability and Computing}, author={Galanis, Andreas and Goldberg, Leslie Ann and Smolarova, Paulina}, year={2025}, pages={359–391}}

@article{RCM-Helmuth2020,
  title={Finite-size scaling, phase coexistence, and algorithms for the random cluster model on random graphs},
  author={Helmuth, Tyler and Jenssen, Matthew and Perkins, Will},
  journal={Annales de l'Institut Henri Poincare (B) Probabilites et statistiques},
  volume={59},
  number={2},
  pages={817--848},
  year={2023}
}

@article{coja2023,
  title={Metastability of the {P}otts ferromagnet on random regular graphs},
  author={Coja-Oghlan, Amin and Galanis, Andreas and Goldberg, Leslie Ann and Ravelomanana, Jean Bernoulli and {\v{S}}tefankovi{\v{c}}, Daniel and Vigoda, Eric},
  journal={Communications in Mathematical Physics},
  pages={1--41},
  year={2023},
  publisher={Springer}
}

@inproceedings{UG1,
author = {Coulson, Matthew and Davies, Ewan and Kolla, Alexandra and Patel, Viresh and Regts, Guus},
title = {Statistical Physics Approaches to {U}nique {G}ames},
year = {2020},
booktitle = {Proceedings of the 35th Computational Complexity Conference},
articleno = {13},
series = {CCC '20}
}

@inproceedings{UG2,
  author       = {Charlie Carlson and
                  Ewan Davies and
                  Nicolas Fraiman and
                  Alexandra Kolla and
                  Aditya Potukuchi and
                  Corrine Yap},
  title        = {Algorithms for the ferromagnetic {P}otts model on expanders},
  booktitle    = {63rd {IEEE} Annual Symposium on Foundations of Computer Science, {FOCS}
                  2022},
  pages        = {344--355},
  year         = {2022}
}

@book{RCMbook,
  title={The random-cluster model},
  author={Grimmett, Geoffrey},
  volume={333},
  year={2006},
  publisher={Springer}
}

@INPROCEEDINGS{Plant1,
  author={Achlioptas, Dimitris and Coja-Oghlan, Amin},
  booktitle={2008 49th Annual IEEE Symposium on Foundations of Computer Science}, 
  title={Algorithmic Barriers from Phase Transitions}, 
  year={2008},
  volume={},
  number={},
  pages={793-802},
  doi={10.1109/FOCS.2008.11}}

@article{Leslie1,
  author       = {Leslie Ann Goldberg and
                  Mark Jerrum},
  title        = {Approximating the partition function of the ferromagnetic {P}otts model},
  journal      = {J. {ACM}},
  volume       = {59},
  number       = {5},
  pages        = {25:1--25:31},
  year         = {2012}
}

@inproceedings{Plant2,
  author       = {Amin Coja{-}Oghlan and
                  Florent Krzakala and
                  Will Perkins and
                  Lenka Zdeborov{\'{a}}},
  title        = {Information-theoretic thresholds from the cavity method},
  booktitle    = {Proceedings of the 49th Annual {ACM} {SIGACT} Symposium on Theory
                  of Computing, {STOC} 2017},
  pages        = {146--157},
  year         = {2017},
}

@article{Plant3,
  author       = {Victor Bapst and
                  Amin Coja{-}Oghlan and
                  Charilaos Efthymiou},
  title        = {Planting Colourings Silently},
  journal      = {Comb. Probab. Comput.},
  volume       = {26},
  number       = {3},
  pages        = {338--366},
  year         = {2017}
}

@inproceedings{Plant1a,
  author       = {Ahmed El Alaoui and
                  Andrea Montanari and
                  Mark Sellke},
  title        = {Sampling from the {S}herrington-{K}irkpatrick {G}ibbs measure via algorithmic
                  stochastic localization},
  booktitle    = {63rd {IEEE} Annual Symposium on Foundations of Computer Science, {FOCS}
                  2022},
  pages        = {323--334},
  year         = {2022}
}

@InProceedings{Plant1b,
  author =	{Efthymiou, Charilaos},
  title =	{{On Sampling Symmetric Gibbs Distributions on Sparse Random Graphs and Hypergraphs}},
  booktitle =	{49th International Colloquium on Automata, Languages, and Programming (ICALP 2022)},
  pages =	{57:1--57:16},
  year =	{2022},
  volume =	{229}
}

@article{Blanca2,
	title = {Sampling in uniqueness from the {P}otts and random-cluster models on random regular graphs},
	author = {Antonio Blanca and Andreas Galanis and Goldberg, {Leslie Ann} and Daniel \v{S}tefankovi\v{c} and Eric Vigoda and Kuan Yang},
	year = {2020},
	volume = {34},
	pages = {742--793},
	journal = {SIAM Journal on Discrete Mathematics},
	number = {1}	
}

@article{BG,
	title = {Random-cluster dynamics on random regular graphs in tree uniqueness},
	author = {Antonio Blanca and Reza Gheissari},
	year = {2021},
	volume = {386},
	pages = {1243--1287},
	journal = {Communications in Mathematical Physics},
	number = {2}	
}

@article{bapst2017planting,
  title={Planting colourings silently},
  author={Bapst, Victor and Coja-Oghlan, Amin and Efthymiou, Charilaos},
  journal={Combinatorics, Probability and Computing},
  volume={26},
  number={3},
  pages={338--366},
  year={2017},
  publisher={Cambridge University Press}
}

@inproceedings{huang2025weak,
  title={Weak Poincar{\'e} inequalities, simulated annealing, and sampling from spherical spin glasses},
  author={Huang, Brice and Mohanty, Sidhanth and Rajaraman, Amit and Wu, David X},
  booktitle={Proceedings of the 57th Annual ACM Symposium on Theory of Computing \emph{(STOC'25)}},
  pages={915--923},
  year={2025}
}

@inproceedings{Chen,
  author       = {Zongchen Chen and
                  Andreas Galanis and
                  Daniel \v{S}tefankovi\v{c} and
                  Eric Vigoda},
  title        = {Sampling colorings and independent sets of random regular bipartite
                  graphs in the non-uniqueness region},
  booktitle    = {Proceedings of the 2022 {ACM-SIAM} Symposium on Discrete Algorithms,
                  {SODA} 2022},
  pages        = {2198--2207},
  publisher    = {{SIAM}},
  year         = {2022},
}

@book{JLR,
	title={Random Graphs},
	author={Janson, Svante and Rucinski, Andrzej and {\L}uczak, Tomasz},
	year={2011},
	publisher={John Wiley and Sons}
}

\begin{appendix}

\section{Remaining Proofs for Section~\ref{sec:prelims}}\label{sec:prelimsproof}
In this section, we give the proof of  Lemma~\ref{lem:quietplanting}, restated here.
\begin{Lemquiet}
\statelemquiteplanting
\end{Lemquiet}
\begin{proof}[Proof of Lemma~\ref{lem:quietplanting}]
 We give the proof for the first item, the proof for the other is completely analogous. Fix arbitrary positive $\beta\leq \beta_c$  and let $\cQ$ be the set consisting of $G\in \cG_{n,d}$  that satisfy \[\Zdis(G)\geq  (f(n))^{1/2}\, \Ex[\Zdis(\G)]\] and note that by Lemma~\ref{lem:smallgraph} it holds that $\Pr[\G\in \cQ]=1-o(1)$.
Moreover, let $\cQ'$ be the set of $G\in \cG_{n,d}$  such that  the set of configurations $S_{\mathrm{Bad}}(G):=\big\{\sigma\in \Sdis\,\big|\, (G,\sigma)\in \mathcal{E}\big\}$ has aggregate weight $Z_{\mathrm{Bad}}(G)=\sum_{\sigma\in S_{\mathrm{Bad}}(G)}\emm^{\beta m_G(\sigma)}$ less than $\frac{g(n)}{f(n)}\Zdis(G)$. Note that for any graph $G\in \cQ'$ we have $\mu_{G,{\dis}}\big((G,\sigma)\in \mathcal{E}\big)=\frac{Z_{\mathrm{Bad}}(G)}{\Zdis(G)}\leq \frac{g(n)}{f(n)}$, so 
to finish the proof, it suffices to show that $\Pr[\G\in  \cQ']=1-o(1)$.

To do this, consider the aggregate weight $W:=\sum_{(G,\sigma)\in  \cE} \emm^{\beta m_G(\sigma)}\mathbf{1}\{\sigma\in \Sdis\}$  over  pairs $(G,\sigma)$ that  belong to $\cE$. By restricting to graphs $G$ in $\cQ\backslash \cQ'$, we have the  lower bound
\begin{equation*}
    W\geq \sum_{G\in \cQ\backslash \cQ'} \sum_{\sigma\in S_{\mathrm{Bad}}(G)}\emm^{\beta m_G(\sigma)}\geq \frac{g(n)}{f(n)}\sum_{G\in \cQ\backslash \cQ'}\Zdis(G).
\end{equation*}
For graphs $G\in\cQ$ we have $\Zdis(G)\geq (f(n))^{1/2}\,\Ex[\Zdis(\G)]$, and therefore we further have
\begin{equation}\label{eq:lowa35}
W\geq \frac{g(n)}{\sqrt{f(n)}} \big|\cQ\backslash \cQ'\big|\ \Ex\big[\Zdis(\G)\big]=\frac{g(n)}{\sqrt{f(n)}} \big|\cQ\backslash \cQ'\big|\,\frac{\sum_{G\in \cG_{n,d}}\sum_{\sigma\in [q]^n} \emm^{\beta m_G(\sigma)}\mathbf{1}\{\sigma\in \Sdis\}}{\big|\cG_{n,d}\big|}
\end{equation}
From the lemma assumption that $\Pr\big[\big(\hat\G(\sdis),\sdis\big)\in  \cE\big]\leq g(n)$, the  definition of $\big(\hat\G\big(\sdis\big),\sdis\big)$ yields that
\begin{equation*}
g(n)\geq \Pr\big[\big(\hat\G(\sdis),\sdis\big)\in  \cE\big]=\frac{\sum_{(G,\sigma)\in \cE}\emm^{\beta m_G(\sigma)}\mathbf{1}\{\sigma\in \Sdis\}}{\sum_{G\in \cG_{n,d}}\sum_{\sigma\in [q]^n} \emm^{\beta m_G(\sigma)}\mathbf{1}\{\sigma\in \Sdis\}}.
\end{equation*}
Combining this with \eqref{eq:lowa35} and the definition of $W$, we obtain $\Pr[\G\in \cQ\backslash \cQ']=\frac{|\cQ\backslash \cQ'|}{|\cG_{n,d}|}\leq (f(n))^{1/2}
=o(1)$. Since $\Pr[\G\in \cQ]=1-o(1)$ from Lemma~\ref{lem:smallgraph}, it follows that 
\[\Pr[G\in \cQ']\geq \Pr[\G\in \cQ\cap \cQ']= \Pr[\G\in \cQ]-\Pr[\G\in \cQ\backslash \cQ']\geq 1-o(1).\qedhere\]
\end{proof}

\section{Remaining Proofs for Section~\ref{sec:RCmixing}}\label{sec:f4ttvrv}
In this section we prove Lemma~\ref{lem:PottstoRC}. The following lemma will be handy.
\begin{lemma}\label{lem:percstep}
Let $d,q\geq3$ be integers and $\beta>0$ be real.  Then, whp over $G\sim\cG_{n,d}$, with $\SIGMA\sim \mu_G$ and $\F=\F(G,\SIGMA)$,  we have  
\begin{equation*}
\begin{aligned}
\mbox{ for $\beta\leq \beta_c$:}&\quad\Pr[\F\in \Omega_{G,\dis}\mid \SIGMA\in \Sdis]=1-\emm^{-\Omega(n)} \quad\mbox{and}\quad\Pr[\F\in \Omega_{G,\ord}\mid \SIGMA\in \Sdis]=\emm^{-\Omega(n)},\\
\mbox{ for $\beta\geq \beta_c$:}&\quad\Pr[\F\in \Omega_{G,\ord}\mid \SIGMA\in \Sord]=1-\emm^{-\Omega(n)} \quad\mbox{and}\quad\Pr[\F\in \Omega_{G,\dis}\mid \SIGMA\in \Sord]=\emm^{-\Omega(n)}.
\end{aligned}
\end{equation*}
\end{lemma}
\begin{proof}
We prove the first equation, the second is analogous. Fix $\beta\leq \beta_c$. By Lemma~\ref{lem:statistics}, for the disordered planted model $\Pr\big[\big\|\rho^{\hat\G(\sdis),\sdis}-\rhodis\big\|_1=o(1)\big]=1-\emm^{-\Omega(n)}$. So Lemma~\ref{lem:quietplanting} yields that $\Pr\Big[\big\|\rho^{\G,\SIGMA_{\G,\dis}}-\rhodis\big\|_1=o(1)\,\Big| \,\G\Big]=1-\emm^{-\Omega(n)}$ whp over $\G$. 
Therefore, for such $\G$, conditioned on $\SIGMA\in \Sdis$, the number of monochromatic edges under $\SIGMA$ is at most $n\big(\tfrac{d}{2}\sum_{i\in [q]}\rho_{\dis,ii}\big)+n\rho / 4 = n\frac{1}{p}\mdis(\beta) +n\rho / 4$ w.p. $1-\emm^{-\Omega(n)}$. Therefore after the percolation step, by standard Chernoff bounds, we have $|\F|\leq n\mdis(\beta) + n\rho / 4$ w.p. $1-\emm^{-\Omega(n)}$. Since $\mdis(\beta)\leq \mdis(\beta_c)$, by the Definition~\ref{def:RCorddis} of $\Omegadis,\Omegaord$  we conclude that $\Pr[\F\in \Omegadis\mid \SIGMA\in \Sdis]= 1-\emm^{-\Omega(n)}$ and $\Pr[\F\in \Omegaord\mid \SIGMA\in \Sdis]= \emm^{-\Omega(n)}$.
\end{proof}
We are now ready to prove Lemma~\ref{lem:PottstoRC}, restated here for convenience. Recall that, for a graph $G$, $\hat\pi_{G,\dis}$ is the distribution on subsets of $E(G)$ obtained by first sampling $\SIGMA_{G,\dis}\sim \mu_{G,\dis}$ and then keeping each monochromatic edge with probability $p=1-\emm^{-\beta}$ (and similarly for $\hat\pi_{G,\ord}$). 
\begin{LemPottstoRC}
\statelemPottstoRC
\end{LemPottstoRC}
\begin{proof}[Proof of Lemma~\ref{lem:PottstoRC}]
We prove the first claim, the proof of the second is analogous. We consider first the case $\beta\leq \beta_c$.  We have the lower bound
\begin{equation}\label{eq:perct5tgtg}
\pi_{G}(\Omega_{G,\dis})\geq \Pr[\SIGMA\in \Sdis]\, \Pr[\F\in \Omega_{G,\dis}\mid \SIGMA\in \Sdis].
\end{equation}
Similarly by conditioning on whether $\SIGMA\in \Sdis$, $\SIGMA\in \Sord$, or $\SIGMA\notin \Sdis\cup \Sord$, we have the inequality
\begin{equation}\label{eq:4tg4t5ygyy}
\pi_{G}(\cA \cap  \Omega_{G,\dis})\leq \Pr[\SIGMA\in \Sdis]\, \Pr[\F\in \cA \mid \SIGMA\in \Sdis]+\Pr[\F\in \Omega_{G,\dis}\mid \SIGMA\in \Sord]+\Pr[\SIGMA\notin \Sdis\cup \Sord].
\end{equation}
For $\beta\leq \beta_c$,  by Lemma~\ref{lem:smallgraph} we have $\Pr[\SIGMA\in \Sdis]\geq \frac{1}{n}$ and $\Pr[\SIGMA\notin \Sdis\cup \Sord]=\emm^{-\Omega(n)}$; by Lemma~\ref{lem:percstep} we also have 
\begin{equation}\label{eq:perc123}
\Pr[\F\in \Omega_{G,\dis}\mid \SIGMA\in \Sdis]=1-\emm^{-\Omega(n)} \quad\mbox{and}\quad\Pr[\F\in \Omega_{G,\dis}\mid \SIGMA\in \Sord]=\emm^{-\Omega(n)}.
\end{equation}
Combining the bounds from \eqref{eq:perct5tgtg}, \eqref{eq:4tg4t5ygyy} and \eqref{eq:perc123}, we obtain that 
\[\pi_{G,\dis}(\cA)= \frac{ \pi_{\G}(\cA \cap  \Omega_{G,\dis})}{ \pi_{G}(\Omega_{G,\dis})}\leq \Pr[\F\in \cA \mid \SIGMA\in \Sdis]+\emm^{-\Omega(n)}=\hat\pi_{\G,\dis}(\cA)+\emm^{-\Omega(n)}.\qedhere\]
\end{proof}

We will also use the following corollary later on.
\begin{corollary}\label{cor:boundary-of-phases-are-unlikely}
Let $d,q\geq 3$ be integers and $\beta>0$ be real. Then, for all $\delta>0$, the following hold whp over $G\sim \cG_{n,d}$.

If $\beta\leq \beta_c$, for $\F_\dis\sim \pi_{G,\dis}$, it holds that $\Pr(|\F_\dis|\geq n\mdis+n\delta)=\emm^{-\Omega(n)}$; if $\beta\geq \beta_c$, for $\F_\ord\sim \pi_{G,\ord}$, it holds that $\Pr(|\F_\ord|\leq n\mord-n\delta)=\emm^{-\Omega(n)}$.
\end{corollary}
\begin{proof}
This follows by applying Lemma~\ref{lem:PottstoRC} and using the percolation step analysis from Lemma~\ref{lem:percstep}. In particular, the latter gives that, for $\hat\F_\dis\sim \hat\pi_{G,\dis}$, it holds that $\Pr(|\hat\F_\dis|\geq n\mdis+n\delta)=\emm^{-\Omega(n)}$. Therefore, from item (i) of  Lemma~\ref{lem:PottstoRC}, we have $\Pr(|\F_\dis|\geq n\mdis+n\delta)=\emm^{-\Omega(n)}$ as wanted. The proof for $\beta\geq \beta_c$ is analogous.
\end{proof}

\section{Remaining Proofs for Section~\ref{sec:disordered-regime}}\label{sec:appendix_proof4.3}

In this section we prove Lemma~\ref{lem:2l-decay-bound-for-twopath}. Throughout this proof, a \emph{nontrivial boundary component of $\eta$} (where \(\eta\) is a boundary condition induced by \(F\setminus B_\ell(v)\) for some \(F\subseteq E\)) means a a component of $(V,F\backslash E(B_\ell(v)))$ containing at least two vertices of $S_\ell(v)$. 

\begin{Lemdecay}
\statelemdecay
\end{Lemdecay}
\begin{proof}[Proof of Lemma~\ref{lem:2l-decay-bound-for-twopath}]
    First of all, it is well-known that whp \(G\sim\cG_{n,d}\) is \emph{locally treelike} (see e.g. \cite[Fact 2.3]{BG}), i.e. that there exists a constant \(L = L(\delta, d)\) such that for all \(v\), we can remove at most \(L\) edges from \(B_\ell(v)\) to make it a tree. Second, by Lemma~\ref{lem:dispath}, there exists \(K = K(d,q,\beta,\delta)\) such that whp over \(G\sim\cG_{n,d}\), with probability \(\geq 1-n^{-3}\), for all \(v\), a random disordered configuration is \(K\)-shattered at distance \(\ell\) from \(v\). Thus from now on we assume \(G\) satisfies these two properties.

    Fix a vertex \(v\) and an edge \(e\) incident to it. With probability \(\leq n^{-3}\), a disordered configuration is not \(K\)-shattered at distance \(\ell\) from \(v\), which gives the \(n^{-3}\) term in the bound. Now assume a configuration is \(K\)-shattered and let \(\eta\) be a boundary condition on \(B_\ell(v)\). To finish the proof, it is sufficient to show there is such \(C\) such that \( \TV{\pi_{G,v,\ell}^-(1_e)-\pi_{G,v,\ell}^\eta(1_e)} \leq C \hat p^{2l}\).

    For the case \(\beta < \beta_u\), this bound was proven in \cite[Proposition 3.3]{BG}, however parts of their proof use estimates that only hold in uniqueness. We are able to extend their bound for all \(\beta\) by closely following the start of their argument and then using different estimates which continue to hold  beyond the uniqueness threshold. 
    
    The starting point for the proof is two results from \cite{BG}: Lemma 5.3 and equation~(5.4).
    Lemma 5.3 bounds the TV-distance on marginals on \(e\) by \(\pi_{G,v,\ell}^\eta(\Upsilon_{G, v,\ell}^\eta)\), where \(\Upsilon_{G,v,\ell}^\eta\) is the event such that a configuration \(\F\) satisfies the following: for all \(1\leq i\leq \ell\), there are two \emph{distinct} vertices \(u_1,u_2\in S_i(v)\) and two corresponding paths \(\gamma_{i,1},\gamma_{i,2}\) such that for both \(j\in\{1,2\}\):

    \begin{itemize}
        \item \(\gamma_{i,j}\) is a path of edges in \(\F\) from \(u_j\) to a vertex of \(S_\ell(v)\) belonging to a nontrivial boundary component of \(\eta\), and
        \item vertices of \(\gamma_{i,j}\) are contained in \(\bigcup_{i\leq i' \leq\ell} S_{i'}(v)\).
    \end{itemize}

    Then in the proof of \cite[Proposition~3.3]{BG}, \(\pi_{G,v,\ell}^\eta(\Upsilon_{G,v,\ell}^\eta)\) is further bound by a probability of an event roughly corresponding to two disjoint paths from a neighbour of \(v\) to \(S_\ell(v)\). In order to define this event, we need to introduce more notation.

    Since \(B_\ell(v)\) is \(L\)-treelike, there is a set \(E_\cE\subseteq E(B_\ell(v))\) of at most \(L\) edges such that removing \(E_\cE\) makes the ball a tree. Thus there are also \(k\leq 2L\) distances \(1\leq d_1<\dots<d_k\leq\ell\) such that \(S_{d_i}(v)\) contains a vertex from \(V(E_\cE)\). Also denote \(d_0 = 1, d_{k+1} = \ell\). The following bound is shown in \cite[Equation (5.4)]{BG}:
    \begin{equation}\label{eq:bound-on-two-path-from-BG}
        \pi_{G,v,\ell}^\eta(\Upsilon_{G,v,\ell}^\eta) \leq K^2 (2L)^{2L} (2^L)^{2L} \sup_{(\Gamma;\Gamma')} \pi^\eta_{G,v,\ell}(\Gamma\cup \Gamma' \in\cdot),
    \end{equation}
    where the supremum is over all pairs of path sequences \((\Gamma;\Gamma') = (\gamma_0,\dots,\gamma_k; \gamma_0',\dots,\gamma_k')\) such that for each \(i=0,\dots,k\), \(\gamma_i\), \(\gamma_i'\) are disjoint paths contained in \(\bigcup_{d_i < j < d_{i+1}} S_j(v)\), such that they can be both extended, by edges in \(\bigcup_{j > d_i}^\ell\) to a path to a nontrivial boundary component of \(\eta\) (note that \(\gamma_i,\gamma'_i\) are allowed to be empty). 
    
    So,  it is enough to show that  for all such \((\Gamma;\Gamma')\), \(\pi^\eta_{G,v,\ell}(\Gamma\cup \Gamma' \in\cdot)\leq C\hat p^{2l}\) for some constant $C>0$. Let \((\Gamma;\Gamma')\) be any such pair of two path sequences. Consider the distribution on the ball~$B_\ell(v)$ with the free boundary and condition further on all excess edges in \(E_\cE\) being out; for brevity denote this distribution by~\(\tilde\pi\). Since removing \(E_\cE\) from the ball makes it a tree, \(\tilde\pi\) is the RC model on a tree, which is the same distribution as performing i.i.d. edge percolation with probability \(\hat p\). Furthermore, \(\Gamma\cup\Gamma'\) does not contain any edges in \(E_\cE\), thus it is an event with positive probability under \(\tilde\pi\). Also, \(\Gamma\cup\Gamma'\) contains at least \(\sum_{i=0}^k 2[d_{i+1}-d_i-2] \geq 2\ell - 8L\) edges, therefore \(\tilde\pi(\Gamma\cup\Gamma')\leq\hat p^{2\ell - 8L}\).

    Since \(\hat p^{-8L}\) is a constant, as a final step it suffices to show that \(\frac{\pi^\eta_{G,v,\ell}(\Gamma\cup\Gamma'\in\cdot)}{\tilde\pi(\Gamma\cup\Gamma'\in\cdot)}\) is bounded by a constant (not depending on \((\Gamma;\Gamma')\)).     Consider any configuration \(A\subseteq E(B_\ell(v))\). Since \(|E_\cE|\leq L\), \(\eta\) is \(K\)-shattered, and adding an edge can decrease the number of components by at most one, we get that \(|A\setminus E_\cE| \leq |A| \leq |A\setminus E_\cE| + L\), and the number of components in the ball for configuration \(A\) w.r.t. boundary condition \(\eta\) differ by at most \(K + L\) from the number of components for configuration \(A\setminus E_\cE\) w.r.t. to the free boundary condition. Thus the weights of \(A\) w.r.t. \(\eta\) and of \(A\setminus E_\cE\) w.r.t. to \(-\) differ by a factor of at most \(q^{L+K} [\max(\emm^\beta-1, (\emm^\beta - 1)^{-1})]^L\).

    Given that there are at most \(2^{|E_\cE|}\leq 2^L\) configurations \(A\subseteq E(B_\ell(v))\) agreeing on \(A\setminus E_\cE\), we get that the cumulative weights of the event ``\(\Gamma\cup\Gamma'\) are all in'' differ by a constant factor (independent of \((\Gamma;\Gamma')\)) when w.r.t. \(\eta\), and w.r.t. \(-\) when restricted to all edges in \(E_\cE\) being out, and so do their respective partition functions.    It hence follows that \(\frac{\pi^\eta_{G,v,\ell}(\Gamma\cup\Gamma'\in\cdot)}{\tilde\pi(\Gamma\cup\Gamma'\in\cdot)}\) is bounded by a constant, which concludes the proof.
\end{proof}

\section{Proof of Lemma~\ref{lem:giantperc}}
\label{sec:appendix_proof5.1}

In this section we give the proof of Lemma~\ref{lem:giantperc}. We first consider the case where $\vec d=(d_1,\hdots,d_n)$ is the degree sequence of a $d$-regular graph (i.e., $d_i=d$ for all $i$). For simplicity, in this uniform setting we denote the percolation model with $m = pdn/2 + o(n)$ edges by $\tilde\G_p$. The following bounds were shown in~\cite{coja2023}. 

\begin{lemma}[{\cite[Proposition 5.4]{coja2023}}]\label{lem:giantperc-regular}
Let $d\geq 3$ be an integer, $p\in (\frac{1}{d-1},1)$, and $m = pdn / 2 + o(n)$. Then, for any $\delta>0$, there exists a constant \(M > 0\), such that with probability $1-\emm^{-\Omega(n)}$ over $\tilde\G_p$, it holds that $\Pr[\cC_1(\tilde\G_p)=(\chi\pm \delta)n]\geq 1-\emm^{-\Omega(n)}$  where $\chi$ is as in \eqref{eq:phi}. Moreover, \(\sum_{i\geq 2} |\cC_i(\tilde\G_p)|^2 \leq Mn\).
\end{lemma}

We use Lemma~\ref{lem:giantperc-regular} to infer bounds for the component sizes for slightly non-regular degree sequences $\vec d$.
\begin{lemma}\label{lem:graph-with-almost-regular-deg-sequence-has-good-components}
    Let \(d\geq 3\) be an integer, \(p \in (1/(d-1), 1)\), \(\delta, \epsilon > 0\). Suppose \(S\subseteq[n]\) satisfies \(|S| = O(n^{1/6})\) and \(\vec d = (d_1,\dots, d_n)\) is a degree sequence such that \(d_i = d\) for \(i\neq S\) and \(d_i\in [0,d]\) for \(i\in S\). Then, there exists a constant \(M(d, p,\delta) > 0\), such that for \(m = \frac{1}{2}dpn + o(n)\), it holds that
    \begin{equation*}
    \Pr\big(\mbox{$|\cC_1(\tilde\G_{n,\vec d,m})| \geq (\chi \pm \delta) n$ and $|\cC_2(\tilde\G_{n,\vec d,m})| \leq M n^{2/3}$}\big) \geq 1 - \emm^{-\Omega(n)}.
    \end{equation*}
\end{lemma}
\begin{proof}
    Observe that the graph $\tilde\G_{n,\vec d,m}$, conditional on the number of percolated edges from vertices of \(S\) being (some) integer \(k\),  the graph \(\tilde\G_{n,\vec d,m}\setminus S\) is distributed as an exact (\(d\)-regular) edge model with \(n-|S|\) vertices and \(m - k\) edges -- as every set of \(m - k\) edges (submatchings) is equally likely by symmetry. Since \(|S| = O(n^{1/6})\), we trivially have that \(k = O(n^{1/6})\) as well. Thus, for any such \(k\), it holds that  \(\frac{m - k}{dn/2} = p + o(1) - O(n^{-5/6}) > \frac{1}{d-1}\) for sufficiently large $n$. Furthermore, since \(\chi(p)\) is continuous on \((\tfrac{1}{d-1}, 1)\),  it also holds that  \(\chi' := \chi(\frac{m - k}{dn/2}) = \chi(p) \pm \delta / 3\). 
    Applying therefore Lemma~\ref{lem:giantperc-regular} to \(\tilde\G_{n,\vec d,m}\setminus S\), we get that with probability \(1-\emm^{-\Omega(n)}\), its largest component has size  \((\chi'\pm \delta/3)(n - |S|) = (\chi \pm \delta)n\) for \(n\) large enough, and hence so does \(\tilde\G_{n,\vec d,m}\). 

    Moreover, by the second part of Lemma~\ref{lem:giantperc-regular}, the second largest component of $\tilde\G_{n,\vec d,m}\setminus S$  has size $O(n^{1/2})$.  Since \(\tilde\G_{n,\vec d,m}\) has at most \(O(n^{1/6})\) more edges than \(\tilde\G_{n,\vec d,m} \setminus S\),  the non-linear components can get up connected to form a component with size at most \(O(n^{1/2} n^{1/6}) = O(n^{2/3})\). 
\end{proof}

We will also use the following lemma.
\begin{lemma}\label{lem:constant-number-at-rth-gen-unlikely}
Consider a supercritical Galton-Watson process and let $Z_r$ be the size of the surviving population at generation $r$. Then, for any fixed integer $K>0$, $\Pr(Z_r\geq K\mid Z_r>0)= 1-o_r(1)$.
\end{lemma}
\begin{proof}
Let $\phi=\Pr(\mbox{extinction})$ and $q_0$ be the probability of no offspring in one step; since the process is supercritical we have $ \phi,q_0<1$. Suppose for the sake of contradiction that there is $\eps>0$ such that $\Pr(Z_r< K\mid Z_r>0)\geq \eps$ for all $r\geq 0$. Then, we have $\Pr(\mbox{extinction}\mid Z_r>0)\geq \delta$ where $\delta=\eps q_0^K>0$.
Hence, for every $r$ we have
\[\Pr(\mbox{extinction})\geq \delta(1-\Pr(Z_r=0))+\Pr(Z_r=0).\]
As $r$ grows large, we have that $\Pr(Z_r=0)\rightarrow \phi$, so by taking limits in the above inequality we obtain that $\phi\geq \delta(1-\phi)+\phi$, contradicting that $\phi<1$.  
\end{proof}

\begin{corollary}\label{cor:prob-constant-no-survivors-in-BP}
    Let \(0 < d_0 \leq d\) be integers, and \(p\in (\tfrac{1}{d-1})\) be a real. Then consider a Galton-Watson process where the root has offspring distribution \(\Bin(d_0, p)\), and all the following generations have offspring distribution \(\Bin(d, p)\). Let \(Z_r\) be the size of the surviving population at generation \(r\). Then, for any fixed integer \(K>0\), \(\Pr(Z_r \geq K\mid Z_r>0)=1-o_r(1)\).
\end{corollary}
\begin{proof}
    Note that once at second or later generation, the process is supercritical, so then we can apply  Lemma~\ref{lem:constant-number-at-rth-gen-unlikely}. Thus we need to consider the case of the root.

    Note that if \(Z_r > 0\), then there is a non-empty set \(\cC\) of children of the root whose branching process survive \(r-1\) generations. Since their individual branching processes -- denote them \((Z_t^i)_{i\in\cC,t\geq 0}\), are independent, \(\Pr(Z_r < K\mid Z_r > 0, \cS\text{ survive}) \leq \prod_{i\in\cC}\Pr(Z_{r-1}^i < K\mid Z_{r-1}^i > 0) = o_r(1)\) by the previous lemma.
Summing over the possible (non-empty) sets of surviving children, we obtain that the bound holds unconditionally.
\end{proof}

Having all necessary results, we can now prove Lemma~\ref{lem:giantperc}, restated here for convenience.

\begin{Lemgiantperc}
\statelemgiantperc
\end{Lemgiantperc}
\begin{proof}[Proof of Lemma~\ref{lem:giantperc}]
    For Item (1) (the subcritical case), the statement is proved in \cite[Proposition 5.2]{coja2023} for \(S = \emptyset\). Note that while the statement of the proposition is for \(d\)-regular degree sequences, as it is noted therein, it also applies for percolation on a graph  of maximum degree \(d\) (with any degree sequence).

    We thus focus on Item (2). The first part of the claim follows from Lemma~\ref{lem:graph-with-almost-regular-deg-sequence-has-good-components}. It only remains to bound the probability that none of the vertices in \(S'\) belong to the largest component. Let \(R, K > 0\) be sufficiently large constants and \(\epsilon_0, \epsilon_1 > 0\) be sufficiently small, to be specified later.

    First, consider revealing the depth-\(R\) neighbourhoods around vertices of \(S'\) in the underlying configuration model, and then the outcome of the percolation on these edges. We say that the revealed \(R\)-neighbourhood of \(v\in S'\) is ``nice'', if it is acyclic, disjoint from all the other revealed \(R\)-neighbourhoods, and all vertices in it other than \(v\) have degree \(d\).   The following claims says that  there are at least \((1-\epsilon_0)|S'|\) nice neighbourhoods.

    \noindent{\bf Claim~1: }{\sl Let \(\epsilon_0\in(0,1)\) be a constant. Then there exists constant \(C > 0\) (not depending on \(n\)) such that with probability \(n^{-C|S'|}\), there are at least \((1-\epsilon_0)|S'|\) vertices \(v\in S'\) with nice neighbourhoods.}
    
    The proof of the claim is standard and similar to the proof of \cite[Lemma 15]{Blanca2} so we sketch the argument briefly. Consider a BFS revealing process from all the vertices in \(S'\) up to depth \(R\). Say that a ``bad event'' is revealing a match going to a vertex that has been either already seen or belongs to \(S\) (only vertices in \(S\) can have degree \(\neq d\)). Note that the size of each neighbourhood is at most \(|S'| d^R\) (at least for \(R\) large enough), and thus in each revealing step, regardless of previous outcomes, the probability of a bad event occurring is bounded by \(\frac{|S'|d^R + |S|}{n - |S'|d^R - |S|}\leq \frac{2|S|d^R}{n}\) (since \(S = O(n^{1/6})\), for large enough \(n\) the denominator is at least \(n/2\)). Hence we can upper bound the number of bad events by a sum of \(|S'| d^R\) i.i.d. Bernoulli random variables with mean \(\frac{2|S|d^R}{n}\). Then by Chernoff bounds (e.g. \cite[Lemma 16]{Blanca2}), \(\Pr(\#\text{bad events} \geq \epsilon_0|S'|/2) \leq \exp\big\{-\frac{\epsilon_0 |S'|}{2} \log\big(\frac{n}{2e|S|d^R}\big)\big\}\). Since \(2e|S|d^R = O(n^{1/6})\), the probability is at most \((C'n^{-5\epsilon_0/12})^{|S'|}\) for some constant \(C' > 0\) (independent of \(n\)). Note that each bad event makes at most two neighbourhoods not nice, and that this probability is at most \(n^{-C'|S'|}\) for some constant \(C'(C,\epsilon_0) > 0\). Thus with probability at least \(1-n^{-C|S'|}\) there are at least \((1-\epsilon_0)|S'|\) nice neighbourhoods, which concludes the claim.

    Having proved the claim, let \(T\) be the subset of \(S'\) with good neighbourhoods. Further on, we will consider the probability of \(S'\cap \cC_1(\tilde\G_{n,\vec d,m})=\emptyset\) conditional on \(|T| \geq (1-\epsilon_0)|S'|\).

    As a next step in the analysis, we find a ``hidden'' \(d\)-regular percolated graph \(\H\) in the remaining unrevealed graph, with a property that whenever a vertex of \(\H\) \textit{corresponding to a vertex of \(S'\)} belongs to \(\cC_1(\H)\), then also \(S'\cap \cC_1(\tilde\G_{n,\vec d,m})\neq\emptyset\). Thus we can upper bound the probability of \(S'\cap \cC_1(\tilde\G_{n,\vec d,m})=\emptyset\), by the probability that neither of vertices corresponding to \(S'\) in \(\H\) belong to \(\cC_1(\H)\).
    
    For each \(v\in T\), consider the set of corresponding \emph{surviving leaves}: that is, the set of vertices distance-\(R\) from \(v\) (in the underlying percolation model), that are connected to \(v\) after percolation. Let \(L_v\) denote the number of surviving leaves corresponding to \(v\).

    Each of surviving leaves has \(d-1\) endpoints with an unrevealed match. Notice that all of these \(L_v\) leaves are connected in \(\tilde\G_{n,\vec d,m}\), and thus, crucially, all of their endpoints can be seen as connected, or belonging to the same vertex. We can hence ``regroup'' the endpoints for the surviving leaves corresponding to each \(v\in T\) into groups of \(d\) endpoints, thus obtaining \(\lfloor\frac{L_v(d-1)}{d}\rfloor\) new vertices of degree \(d\) (we can ``discard'' the excess endpoints). These new vertices will be the vertices in \(\H\) corresponding to \(v\).

    Thus we have three groups of endpoints / vertices of \(\tilde\G_{n,\vec d,m}\):
    \begin{itemize}
        \item Vertices (and their corresponding endpoints) in \(V\setminus S\) that were not revealed by the \(R\)-neighbourhood exploration process. Call these \(V_0\).

        \item The newly formed vertices corresponding to the ``regrouped'' endpoints: for each \(v\in T\), divide the endpoints corresponding to its surviving leaves into \(\lfloor\frac{L_v(d-1)}{d}\rfloor\) new vertices -- denote these as \(V_v\). 
        
        \item All the remaining endpoints and vertices. In particular, this third group consists of unrevealed endpoints of \(S\) and endpoints of vertices in the \(R\)-neighbourhoods that were not yet revealed, which do not fall into the second category of regrouped vertices.
    \end{itemize}
    The vertices of \(\H\) will be \(V_0\cup\bigcup_{v\in T}V_v\), and all of them have degree \(d\). The edges of \(\H\) are all percolated edges of \(\tilde\G_{n',d,m}\) that are between two endpoints that also belong to a vertex of \(\H\).

    Next we claim, that conditional on the number of vertices in \(\bigcup_{v\in T} V_v\) and the number edges in \(\tilde\G_{n',d,m}\) containing an endpoint in the third category, \(\H\) is distributed as a $d$-regular exact edge model with $n' = n - O(n^{1/6})$ vertices and \(m' = \tfrac 12 dpn + o(n)\) edges. 

    For the number of vertices, it follows from \(|V_0| \leq n' = |V_0| + |\bigcup_{v\in T}V_v|\), and then from \(|V_0| \geq n - |S| - d^R|S'| = n - O(n^{1/6})\), and \(|\bigcup_{v\in T}V_v| = d^R |S'| = O(n^{1/6})\).

    For the number of edges, note that the number of endpoints in the third category is at most \(|S| + |S'| d^{R+1} = O(n^{1/6}) = o(n)\) and the number of already revealed edges is also bound by \(O(n^{1/6})\). Hence we get \(m' = \tfrac 12 dpn + o(n)\).

    For the distribution, we note that the matches of unrevealed endpoints of the third category are uniformly at random, thus by symmetry, all sets of \(m'\) edges of \(\H\) are equally likely, thus it is distributed by exact edge model.

    The following claim implies that we can upper bound the probability of \(S'\cap \cC_1(\tilde\G_{n,\vec d,m}) =\emptyset\) by the probability of \((\bigcup_{v\in T} V_v)\cap \cC_1(\H)\), at least as long as the components of \(\H, \tilde\G_{n,\vec d,m}\) are sufficiently well-behaved (as in the statement of this Lemma). The latter probability is easier to bound given the symmetric nature of \(\H\).

    \noindent{\bf Claim~2: }{\sl  
        Let \(v\in T\), and \(w\in V_v\) be such that \(w\) is in a component of size \(s\) in \(\H\). Then \(v\) is in a component of size \(\geq s - |S'|d^{R}\) in \(\tilde\G_{n,\vec d,m}\).}

    This claim follows from the facts that the endpoints in \(\H\) corresponding to \(w\) correspond, in \(\tilde\G_{n,\vec d,m}\), to vertices that are connected in \(\tilde\G_{n,\vec d,m}\) and in particular are also connected to \(v\), thus any edge of the component of \(w\) is also in the component of \(v\). Then all vertices in \(V_0\) belonging to the component of \(w\) in \(\H\) are also in the component of \(v\) in \(\tilde\G_{n,\vec d,m}\), and \((V(\H)\setminus V_0)\leq |S'|d^R\).

    Let \(\cC_{good}\) be the event that \(\H\) and \(\tilde\G_{n,\vec d,m}\) have components satisfying \(\cC_1(\H)\geq \chi n'/2\), \(\cC_1(\tilde\G_{n,\vec d,m})\geq \chi n / 3\), and \(\cC_2(\tilde\G_{n,\vec d,m})\leq Mn^{2/3}\), where \(M\) is a constant from Lemma~\ref{lem:graph-with-almost-regular-deg-sequence-has-good-components} (taking \(\delta = 2\chi/3\)). We can assume \(n\) is large enough so \(M n^{2/3} < \chi n'/2 - d^R|S'| \geq \chi n / 3\). Note this holds for all sufficiently large \(n\) as \(n' = n - O(n^{1/6})\). By the Lemma~\ref{lem:graph-with-almost-regular-deg-sequence-has-good-components}, \(\Pr(\cC_{good}) \geq 1 - \emm^{-\Omega(n)}\).

    Conditional on \(\cC_{good}\), the above claim implies that whenever \(w\in V_v\) belongs to \(\cC_1(\H)\), then \(v\in \cC_1(\tilde\G_{n,\vec d,m})\). Next we show that for any \(\epsilon_1 > 0\), conditional on \(T\), the following inequality holds:\begin{equation}\label{eq:prob-none-in-giant-of-H}
        \Pr(\forall v\in T. V_v\cap \cC_1(\H) = \emptyset \mid T, \cC_{good}) \leq \prod_{i\in T} [(\varphi p + 1 - p)^{d_i} + \epsilon_1]
    \end{equation}
    In order to prove \eqref{eq:prob-none-in-giant-of-H}, first note that since that \(\H\) is obtained by percolation of u.a.r. \(d\)-regular graph, conditional on size of the largest component, the vertices belonging to it are a u.a.r. subset of that size.

    Write \(T = \{i_1,\dots,i_k\}\), and let \(1\leq j\leq k\). Suppose we conditioned on the ``outcome'' for \(R\)-neighbourhoods of \(i_1,\dots,i_{j-1}\): that is, the number of surviving leaves, and which of the vertices in \(\bigcup_{\ell<j} V_{i_\ell}\) belong to the largest component. Regardless of the previous neighbourhoods, there are at least \(\chi n' / 2 - d^R |S'| \geq \chi n / 3\) undetermined vertices in the largest component, and these are chosen, by symmetry, uniformly at random from the set of at most \(n\) not yet determined vertices.

    Now consider the number of surviving leaves, \(L_{i_j}\), in the \(R\)-neighbourhood of \(i_j\). By Corollary~\ref{cor:prob-constant-no-survivors-in-BP}, the probability that \(L_{i_j} < Kd\) is \((\varphi p + 1 - p)^{d_{i_j}} + o_R(1)\) (as \((\varphi p + 1 - p)^{d_{i_j}}\) is the survival probability for the corresponding branching process). Taking \(R\) large enough, we can get \(o_R(1) \leq \epsilon_1/2\). Then, conditional on \(L_{i_j} \geq Kd\), \(|V_{i_j}|\geq K(d-1)\), and the probability that neither of these \(K(d-1)\) belong to a u.a.r. subset of \(<n\) vertices of size \(\geq \chi n / 3\) is at most \((1-\chi/3)^{K(d-1)}\rightarrow 0\) as we increase \(K\). Thus for \(K\) large enough, this probability is at most \(\epsilon_1/2\).

    Putting it all together, the probability that \(V_{i_j}\cap \cC_1(\H) = \emptyset\) is at most \((\varphi p + 1 - p)^{d_{i_j}} + \epsilon_1\). Since this is regardless of outcomes of the previous neighbourhood, we can take a product, and thus obtain \eqref{eq:prob-none-in-giant-of-H}.
    
    To obtain the final probability, we unravel the conditioning:
    \begin{align*}
    \Pr(S'\cap \cC_1(\tilde\G_{n,\vec d,m}) = \emptyset) &\leq \Pr(S'\cap \cC_1(\tilde\G_{n,\vec d,m}) = \emptyset\mid |T|\geq (1-\epsilon_0)|S'|) + \Pr(|T| < (1-\epsilon_0)|S'|) \\ & \leq \Pr(S'\cap \cC_1(\tilde\G_{n,\vec d,m}) = \emptyset\mid |T|\geq (1-\epsilon_0)|S'|,\cC_{good}) \\ &+ \Pr(\overline{\cC_{good}} \mid |T|\geq (1-\epsilon_0)|S'|) + \Pr(|T| < (1-\epsilon_0)|S'|)
    \end{align*}
    We showed \(\Pr(|T| < (1-\epsilon_0)|S'|) \leq n^{-C|S'|}\), and \(\Pr(\cC_{good}\mid |T|\geq (1-\epsilon_0)|S'|) \leq \emm^{-\Omega(n)}\). \eqref{eq:prob-none-in-giant-of-H} gives \(\Pr(S'\cap \cC_1(\tilde\G_{n,\vec d,m}) = \emptyset\mid |T|\geq (1-\epsilon_0)|S'|, T) \leq \prod_{i\in T} [(\varphi p + 1 - p)^{d_i} + \epsilon_1]\). 
    
    It remains to show that, for \(\epsilon_0,\epsilon_1\) sufficiently small and $|T|\geq (1-\epsilon_0)|S'|$, we have that
    \begin{equation}\label{eq:T-bound-bounded-by-S'-bound}
        \prod_{i\in T} [(\varphi p + 1 - p)^{d_i} + \epsilon_1] \leq (1+\epsilon/2)^{|S'|} \prod_{i\in S'} (\varphi p + 1 - p)^{d_i}
    \end{equation}
    As a first step towards \eqref{eq:T-bound-bounded-by-S'-bound}, note that for each \(i\in S'\), \(d_i\leq d\), and since also \((\varphi p + 1 - p) < 1\), it follows that \((\varphi p + 1 - p)^{d_i} + \epsilon_1\leq (\varphi p + 1 - p)^{d_i}(1 + \epsilon_1 / (\varphi p + 1 - p)^d)\).  Also, we can bound, for all sufficiently large \(T\),
    \[
    \frac{\prod_{i\in T} (\varphi p + 1 - p)^{d_i}}{\prod_{i\in S'} (\varphi p + 1 - p)^{d_i}} = \prod_{i\in S'\setminus T} (\varphi p + 1 - p)^{-d_i} \leq (\varphi p + 1 - p)^{-d|S'\setminus T|}\leq (\varphi p + 1 - p)^{d\epsilon_0|S'|}
    \]
    Combining, we therefore get that
    \[
        \prod_{i\in T} [(\varphi p + 1 - p)^{d_i} + \epsilon_1] \leq [(1 + \tfrac{\epsilon_1}{ (\varphi p + 1 - p)^d})(\varphi p + 1- p)^{d\epsilon_0}]^{|S'|} \prod_{i\in S'} (\varphi p + 1 - p)^{d_i}
    \]
    Note that by taking \(\epsilon_0, \epsilon_1 \rightarrow 0\), we get \((1+\frac{\epsilon_1}{(\varphi p + 1 - p)^d})(\varphi p + 1 - p)^{d\epsilon_0}\rightarrow 1\), so in particular we can have \(\epsilon_0,\epsilon_1 > 0\) small enough so that expression becomes at most \(1 + \epsilon/2\), yielding~\eqref{eq:T-bound-bounded-by-S'-bound}.
    
    Finally, note that \(\prod_{i\in S'} (\varphi p + 1 - p)^{d_i}\leq (\varphi p + 1 - p)^{d|S'|}\), and for all sufficiently large \(n\), \(n^{-C|S'|} + \emm^{-\Omega(n)}\leq [\epsilon(\varphi p + 1 - p)^{d}/2]^{|S'|}\) (as \(|S'| = O(n^{1/6})\), \(n^{-C|S'|} + \emm^{-\Omega(n)} = o_n(1)^{|S'|}\)). Therefore we get that $\Pr(S'\cap \cC_1(\tilde\G_{n,\vec d,m}) = \emptyset)$ is at most
    \[
    [\epsilon(\varphi p + 1 - p)^{d}/2]^{|S'|} + (1+\epsilon/2)^{|S'|} \prod_{i\in S'}(\varphi p + 1 - p)^{d_i} \leq (1 + \epsilon)^{|S'|} \prod_{i\in S'}(\varphi p + 1 - p)^{d_i},
    \]
    which completes the proof.
\end{proof}

\section{Remaining Proofs for Section~\ref{sec:orderedproof}}

\subsection{Proof of Lemma~\ref{lem:ordpath}}\label{sec:E1E1}

In this section, we give the proofs of Equations~\eqref{eq:Ptauxi} and~\eqref{eq:step13223}.

Recall that \(\nu_\ord\) and \(\rho_\ord\) are respectively the vertex colour and the half-edge colour statistic for the ordered phase. Then let \(P\) is an arbitrary sequence of vertices \((v_0,v_1,\dots,v_\ell)\) with \(q\)-colouring \(\tau = (\tau_0,\dots,\tau_\ell\) and \(\xi\) is a set of paired half-edges that turn \(v_0,\dots,v_\ell\) into a path in that order. Then \(\cE_{P,\tau,\epsilon}\) is the event that \(v_0,v_1,\dots,v_\ell\) form a path in \(\hat\G\) using the half-edges prescribed by \(\epsilon\) and that their colours are as prescribed by \(\tau\). We claim the following asymptotic bound on the probability of \(\cE_{P,\tau,\xi}\):
\begin{equation*}
 \Pr(\cE_{P,\tau,\xi})\sim (dn)^{-\ell}\nu_{\ord,\tau_0}\prod^{\ell-1}_{k=0}\frac{\rho_{\ord, \tau_k\tau_{k+1}}}{\nu_{\ord,\tau_k}}.\tag{\ref{eq:Ptauxi}}
\end{equation*}

\begin{proof}[Proof of Equation~\eqref{eq:Ptauxi} ]
    Recall that the graph $\hat \G$ has $ n_i=n\nu_i$ vertices of colour $i\in [q]$ and the number of edges between colour classes $i$ and $j$ with $i\neq j$ is given by $2m_{ij}:=d n \rho_{ij}$ if $i\neq j$, and $m_{ii}:=\frac{d n}{2} \rho_{ii}$ if $i=j$, cf.  \eqref{eq:edgestat}; we further have that $\nu$ and $\rho$ satisfy $\|\nu-\nuord\|_1+\|\rho-\rhoord\|_1=O(n^{-1/3})$. The graph $\hat \G$ is uniformly distributed conditioned on these statistics. In particular, the number of ways to match the half-edges consistent with the vertex statistics $\{n_i\}$ and edge-statistics $\{m_{ij}\}$ is given by 
\[
\mathcal{N}:=\binom{n}{n_1,\dots,n_q}\prod_{i\in[q]}\binom{d n_i}{2m_{i1},\dots,2m_{iq}} \prod_{i\in [q]} \frac{(2m_{ii})!}{ 2^{m_{ii}} m_{ii}!} \times \prod_{i<j} (2m_{ij})!.
\]
 There are \(\binom{n}{n_1,\dots,n_q}\) ways to assign colours to get the right colour statistics. Then, for a colour $i\in [q]$, there are $d n_i$ half-edges which can be split among the $q$-colour classes in 
 $\binom{d n_i}{2m_{i1},\dots,2m_{iq}}$  ways so that the $j$-th class gets $2m_{ij}$ half-edges ($j\in [q]$).  For colours $i \neq j$,  the number of ways to match the selected half-edges is $(2m_{ij})!$, whereas for $i=j$ it is $(2m_{ii})!!=\frac{(2m_{ii})!}{ 2^{m_{ii}} m_{ii}!}$.

For the $\tau$-coloured path $P$, for colours $i,j\in [q]$ with $i\neq j$, let $\ell_i=\ell_i(\tau)$ be the number of vertices of colour $i$, $e_{ii}=e_{ii}(\tau)$ be the number of edges whose endpoints have both colour $i$, and   $2e_{ij}=2e_{ij}(\tau)$  be the number of edges with one endpoint of colour $i$ and the other colour $j$. Recall also that \(\xi\) is a subset of paired half-edges of vertices of \(P\) that turn $P$ into a path. Note that for $i\in [q]$, every vertex $v$ in the path of colour $i$ uses two half-edges from the colour class $i$ (unless $v$ is one of the endpoints $v_0$ and $v_\ell$ in which cases it uses one), so the remaining number of half-edges incident to colour class $i$ is   $M_i=d n_i-2\ell_i+\mathbf{1}\{\tau_0=i\}+\mathbf{1}\{\tau_\ell=i\}$. Let $m'_{ij}=m_{ij}-e_{ij}$.

Akin to the previous counting, we  get that the number of ways to colour remaining vertices and match the remaining edges so that $\cE_{P,\tau,\xi}$ occurs is $\mathcal{N}'$ where
\[
\mathcal{N}':=\binom{n-(\ell+1)}{n_1-\ell_1,\dots,n_q-\ell_q} \prod_{i\in[q]}\binom{M_i}{2m_{i1}',\dots,2m_{iq}'} \times \prod_{i\in [q]} \frac{(2m_{ii}')!}{ 2^{m_{ii}' (m_{ii}')!}}\times \prod_{i<j} (2m_{ij}').
\]
We have
\begin{align}\label{eq:ordered-bound-term-1}
\frac{\mathcal{N}'}{\mathcal{N}}= \frac{(n-(\ell+1))!}{n!}\prod_{i\in[q]}\frac{n_i!}{(n_i-\ell_i)!}\,\prod_{i\in [q]}\frac{(M_i)!}{(dn_i)!}\,\prod_{i\in [q]} \frac{2^{m_{ii}}(m_{ii})!}{2^{m_{ii}'}(m_{ii}')!}\,\prod_{i< j}\frac{(2m_{ij})!}{(2m_{ij}')!}.
\end{align}
We have that $\frac{(n-(\ell+1))!}{n!}\sim n^{-(\ell+1)}$. Using that $n_i=n \nu_{\ord,i}+O(n^{2/3})$ for $i\in [q]$ and $\sum_{i\in [q]} \ell_i=\ell+1$, we also have $\prod_{i\in[q]}\frac{n_i!}{(n_i-\ell_i)!}\sim n^{\ell+1}\prod_{i\in [q]} (\nu_{\ord,i})^{\ell_i} $ 
and $\prod_{i\in [q]}\frac{(M_i)!}{(dn_i)!}\sim \frac{(dn)^2\nu_{\ord,\tau_0}\nu_{\ord,\tau_\ell}}{\prod_{i\in [q]} (dn\,\nu_{\ord,i})^{2 \ell_i}}$; note that the factor $(dn)^2\nu_{\ord,\tau_0}\nu_{\ord,\tau_\ell}$ accounts for the presence of the term $\mathbf{1}\{\tau_0=i\}+\mathbf{1}\{\tau_\ell=i\}$ in the expression for $M_i$.  Using in addition that $2m_{ij}=dn \rho_{\ord,ij}+O(n^{2/3})$, we also have that $\prod_{i\in [q]} \frac{2^{m_{ii}}(m_{ii})!}{2^{m_{ii}'}(m_{ii}')!}\prod_{i< j}\frac{(2m_{ij})!}{(2m_{ij}')!}\sim \prod_{i,j\in [q]} (dn\rho_{\ord,ij})^{e_{ij}}$. Combining these and using that\footnote{To see this, note that every internal vertex of $P$ has degree 2, by considering the edges incident to colour $i$ (and paying attention to endpoints) we have that $2 \ell_i-\mathbf{1}\{\tau_0=i\}-\mathbf{1}\{\tau_\ell=i\}=\sum_{j\in [q]} 2e_{ij}$. Summing over $i\in [q]$ and using that $\sum_i \ell_i=2(\ell+1)$, yields the stated equality.} $\sum_{i,j}e_{ij}=\ell$. 
we obtain that
\[
\Pr(\cE_{P,\tau,\xi})=\frac{\mathcal{N}'}{\mathcal{N}}\sim  (dn)^{-\ell}\nu_{\ord,\tau_0}\nu_{\ord,\tau_\ell} \prod_{i,j\in[q]} (\rho_{\ord,ij})^{e_{ij}} \prod_{i\in[q]} (\nu_{\ord,i})^{-\ell_i}.
\]
It is a matter of a few manipulations now to massage the last expression into \eqref{eq:Ptauxi}.
\end{proof}

We next prove Equation~\eqref{eq:step13223}, i.e. 
\begin{equation*}
\Pr\Big[\begin{array}{c}\mbox{the largest component of $\hat{\G}_v(\F)$}\\\mbox{avoids the $\tau$-coloured path $P$}\end{array}\,\Big|\, \cE_{P,\tau,\xi}\Big]\leq (1+\tfrac{\eps}{3})^{\ell} \hat\varphi^{\ell_1(\tau)}_1,\tag{\ref{eq:step13223}}
\end{equation*}
where $\ell_1(\tau)$ is the number of vertices with colour 1 under $\tau$ and $\hat \varphi_1$ is the extinction probability of the modified Galton-Waston process (where the root has degree \(d-2\), non-root vertices have degree \(d\)), as defined in \eqref{eq:phi}.

\begin{proof}[Proof of Equation~\eqref{eq:step13223}]
    Condition on~$\cE_{P,\tau,\xi}$ and on the set of  edges $E_v$ in~$\hat \G$ incident to $v$. Let  $\mathbf{d}$ be the degree sequence of the graph $\hat\G\backslash v=(V\backslash v, E\backslash (E_P\cup E_v))$. For a colour $i\in[q]$, let $\V_{\nid i}$ denote the $i$-th colour class of $\hat\G\backslash v$ under $\hat\SIGMA_\ord$, $\mathbf{d}_i$ be the subsequence of $\mathbf{d}$ corresponding to vertices in $\V_{\nid i}$, $\E_i$ be the set of monochromatic edges within $\V_{\nid i}$, and $\F_i=\F\cap \E_i$.   Conditioned on $|\F_i|=f_i$, the graph  $\G_i=\big(\V_{\nid i}, \F_i\backslash (E_P\cup E_v)\big)$ has the same distribution as the exact-edge model $\tilde\G_{ n_i,\mathbf{d}_i, f_i}$ where $ n_i=|\V_{\nid i}|=n\nu_{\ord,i}+o(n)$. Note that for all $i\in[q]$, there are at most $\ell+1+d = O(n^{1/6})$ vertices in $\vec d_i$ with degree $\leq d$ (all other vertices have degree $d$).

Since $|\E_i|=\frac{dn}{2}\rho_{\ord,ii}+o(n)$, by a Chernoff bound we have that $|\F_i|=p\frac{dn}{2}\rho_{\ord,ii}+o(n)$ with probability 
$1-\emm^{-\Omega(n)}$. 
It follows from Lemma~\ref{lem:percord}, applied with densities $p_i=p \frac{\rho_{\ord,ii}}{\nu_{\ord,i}}$ for $i\in [q]$, that  ${\G}_1$ is in the supercritical regime while ${\G}_i$ for $i\in \{2,\hdots,q\}$ is in the subcritical regime. 
Hence, it follows from Lemma~\ref{lem:giantperc} that  w.p. $1-\emm^{-\Omega(n)}$, the largest component $\cC$ in $ \G_1$ has size $\Omega(n)$ and the remaining components have all size $O(n^{2/3})$. Moreover,  from Lemma~\ref{lem:giantperc} part (2) we obtain that\footnote{Note that, because the edges in $E_v$ are excluded from $\hat\G\setminus v$,   all but at most $d+2$ vertices of the path $P$  have degree $d-2$ in $\hat\G\setminus v$; the contribution of vertices with degree $<d-2$ is at most a constant which is absorbed into the $(1+\frac{\eps}{3})^\ell$ factor in \eqref{eq:giant1}.}
\begin{equation}\label{eq:giant1}
\Pr[\,\cC\mbox{ avoids the path } P\mid \cE_{P,\tau,\xi}, E_v]\leq (1+\tfrac{\eps}{3})^{\ell}(p\varphi+1-\varphi)^{(d-2)\ell_1}\leq (1+\tfrac\eps 3)^{\ell+1}\hat\varphi_1^{\ell_1}.
\end{equation}
By contrast, for $i\in \{2,\hdots,q-1\}$, w.p. $1-\emm^{-\Omega(n)}$, the largest component in $ \G_i$ has size $O( n^{1/2})$. Since the union of the $ \G_i$'s over $i\in [q]$ is the graph $\widetilde \G_{P,v}=\big(\V\backslash \{v\}, \F\backslash (E_P\cup  E_v)\big)$, by a union bound it follows that, w.p. 
$1-\emm^{-\Omega(n)}$,  
$\cC$ is the largest component in $\widetilde \G_{P,v}$ with $\Omega(n)$ size, while all the others have size $O(n^{2/3})$.
Now add the monochromatic edges of the path $P$, each percolated with probability $p$, to obtain the graph $\hat{\G}_v(\F)=(\V\backslash \{v\},\F\backslash \F_v)$. Let $\cC'$ be the component of $\hat{\G}_v(\F)$ with $\cC\subseteq \cC'$. Since we added only edges on the path $P$, we have
\begin{equation}\label{eq:giant2}\Pr[\,\cC'\mbox{ avoids the path } P\mid \cE_{P,\tau,\xi}]=\Pr[\,\cC\mbox{ avoids the path } P\mid \cE_{P,\tau,\xi}].
\end{equation}
Moreover, the length of  the path $P$ is $\leq n^{1/6}$, so $|\cC'|-|\cC|=O(n^{5/6})$. Similarly, any component in $\hat{\G}_v(\F)$ other than $\cC'$ has size $O(n^{2/3})$, and in particular $\cC'$ is the largest component in $\hat{\G}_v(\F)$.  Combining this with \eqref{eq:giant1} and \eqref{eq:giant2}, we obtain \eqref{eq:step13223}.
\end{proof}

\subsection{Remaining proofs for Lemma~\ref{lem:ordered-less-than-(d-1)^-3}}\label{sec:proof-lemma-18}

Next we prove Lemma~\ref{lem:ordered-less-than-(d-1)^-3} which we restate for convenience.

\begin{LemmaOrderedUpperBound}
\stateOrderedUpperBound
\end{LemmaOrderedUpperBound}
\begin{proof}
First, we observe that for $\beta=\beta_c$, it holds that $\emm^{\beta}=\frac{q-2}{(q-1)^{1-2/d}-1}$ and hence the value of $t$ in \eqref{eq:nuord} is $t_c=(q-1)^{2/d}$, which in turn gives $\nu_{\ord,1}=\frac{q-1}{q}$ and $\nu_{\ord,i}=\frac{1}{q(q-1)}$ for $i\in \{2,\hdots,q\}$. Consider henceforth arbitrary $\beta\geq \beta_c$; then the value of $t$ in \eqref{eq:nuord} satisfies $t\geq t_c$.

To show the desired inequality, we use the expressions for $\rho_{\ord,ij}$ given in \eqref{eq:rhostatistics}. We have that $\frac{\rho_{\ord,i1}}{\nu_{\ord,i}}$ is $\emm^{\beta}\frac{(\nu_{\ord,1})^{(d-1)/d}}{(\nu_{\ord,1})^{1/d}}/W$ if $i=1$ and $\frac{(\nu_{\ord,1})^{(d-1)/d}}{(\nu_{\ord,i})^{1/d}}/W$ for $i\neq 1$, where $W=\sum_{i,j\in [q]}\emm^{\beta\mathbf{1}\{i=j\}}(\nu_{\ord,i}\nu_{\ord,j})^{(d-1)/d}$. From \eqref{eq:nuord}, we have that   $\emm^\beta \big(\frac{\nu_{\ord,i}}{\nu_{\ord,1}}\big)^{1/d}=\emm^\beta/t\geq 1$, so $\frac{\rho_{\ord,11}}{\nu_{\ord,1}}\geq \frac{\rho_{\ord,i1}}{\nu_{\ord,i}}$ for all $i\in\{2,\hdots,q\}$. It suffices therefore to show the bound for $i \in \{2,\hdots,q\}$; fix an arbitrary such $i$.

From \eqref{eq:nuord}, we have that $\nu_{\ord,1}=\frac{t^d}{t^d+q-1}$, $\nu_{\ord,i}=\frac{1}{t^d+q-1}$ 
and hence the expression for $\rho_{\ord,i1}$ given in \eqref{eq:rhostatistics} 
yields $\rho_{\ord,{i1}}=\frac{t^{d-1}}{(t^{d-1}+q-1)^2+(\emm^{\beta}-1)(t^{2d-2}+(q-1))}$. Using the expression for $\emm^{\beta}-1$ from \eqref{eq:nuord}, it follows after some algebra that
\begin{equation}\label{eq:Rexprs12}\frac{\rho_{\ord,{i1}}}{\nu_{\ord,i}}=1-R \mbox{ where } R:=\tfrac{   t^{d+1}+(q-2)t^d-(q-1)t}{(t^d - t) (t^{d-1}+q-1)}
\end{equation}
So, to show the inequality $(1-\sqrt{\hat\varphi_1})\frac{\rho_{\ord,i1}}{\nu_{\ord,i}}>1-\frac{1}{(d-1)^5}$ it suffices to show that $\frac{1}{(d-1)^5}>\sqrt{\hat\varphi_1}+R$.

To upper bound $\hat\varphi_1$, we need first to find the solution of $\varphi_1=(p_1 \varphi_1+1-p_1)^{d-1}$ for $p_1=p\frac{\rho_{\ord,11}}{\nu_{\ord,1}}$. Similarly to above, we have $\rho_{\ord,11}=\frac{\emm^\beta t^{2(d-1)}}{(t^{d-1}+q-1)^2+(\emm^{\beta}-1)(t^{2d-2}+(q-1))}$, so using the expression for $\emm^{\beta}-1$ from \eqref{eq:nuord} we obtain after some algebra that $p_1=\frac{t^{d-1}(t-1)}{t^{d}-t}$ and $\varphi_1=1/t^{d-1}$. Since  $t\geq t_c$ for $\beta\geq \beta_c$, we obtain that $\varphi_1\leq \frac{1}{(q-1)^{2(d-1)/d}}$. It follows that $\hat\varphi_1=(p_1 \varphi_1+1-p_1)^{d-2}=\varphi_1^{(d-2)/(d-1)}\leq \frac{1}{(q-1)^{2(d-2)/d}}$, so $\sqrt{\hat\varphi_1}\leq \frac{1}{(q-1)^{(d-2)/d}}$. For $q\geq (5d)^5$, numerical calculations give that this is less than $\frac{1}{2(d-1)^5}$ for all $d\geq 3$ (the inequality is tighter for small values of $d$). A similar argument shows that $R<\frac{1}{2(d-1)^5}$ for all $d\geq 3$ when $q\geq (5d)^5$;  to see this, we have from the expression in \eqref{eq:Rexprs12} that $R\leq R'$ where $R':=\tfrac{t+q-1}{t^{d-1}+q-1}$. We have that $R'$ is decreasing in $t$, so using that $t\geq t_c$, we obtain that $R'\leq \tfrac{t_c+q-1}{t^{d-1}_c+q-1}=\frac{1}{(q-1)^{(d-2)/d}}$; from the same numerical inequality as above, we therefore conclude that $R$ is less than $\frac{1}{2(d-1)^5}$ for all $d\geq 3$ when $q\geq (5d)^5$.   Hence, $\frac{1}{(d-1)^5}>\sqrt{\hat\varphi_1}+R$, as needed.
\end{proof}

\subsection{Remaining Proofs for Section~\ref{sec:wsm-ordered}}\label{sec:wired-boundary-proofs}

Finally, we give the proofs of Lemmas~\ref{lem:wsm-implied-from-wired-boundary} and~\ref{lem:no-long-path-outside-giant=>wired}, which we restate for convenience.

\begin{LemmaTVatmostWB}
\stateTVastmostWB
\end{LemmaTVatmostWB}
\begin{proof}
    Let \(\F_\ord\sim\pi_{G,\ord}\) and \(\F^+\sim\pi_{G,v,\ell}^+\).

    Let \(\cO\) be the event that \(\F_\ord\) contains at least \(n\mord(\beta_c)-n\rho /4\) edges outside of the ball (where \(\rho=\mord(\beta_c)-\mdis(\beta_c)\) as in Definition~\ref{def:RCorddis}). Since \(|E(B_\ell(v))| = O(n^{1/2})\),  for \(n\) large enough we have \(\Pr(\overline{\cO})\leq \Pr(|\F_\ord|\leq n\mord - n\rho/8)\), which is by Corollary~\ref{cor:boundary-of-phases-are-unlikely}, whp over \(G\), bounded by \(\emm^{-\Omega(n)}\). Note that conditional on \(\cO\), regardless of the state of the edges in the ball $B_{\ell}(v)$, the resulting configuration is ordered and thus we can thus drop the conditioning on the ordered phase. 

    By the monotonicity property of the random cluster model (see e.g. \cite[Appendix B]{RCrandom}), we get that \(\pi_{G,v, \ell}^+\) stochastically dominates \(\pi_{G,\ord}(\cdot\mid\cO)=\pi_G(\cdot\mid \cO)\) on \(B_\ell(v)\), hence conditional on \(\cO\) we can perfectly couple \(\F_\ord\sim\pi_{G,\ord}\) and \(\F^+\sim\pi_{G,v, \ell}^+\) such that \(\F_\ord\cap E(B_\ell(v)) \subseteq \F^+ \cap E(B_\ell(v))\). Then we can bound the probability that $\F_\ord$ and $\F^+$  disagree on \(e\) by 
    \begin{align*}
        \Pr(\F_\ord,\F^+\text{ disagree on }e) \leq&\, \Pr(\F_\ord\text{ does not have a wired boundary}) + \Pr(\overline{\cO}) \\ &+   \Pr(e\not\in \F_\ord, e\in \F^+, \F_\ord\text{ has a wired boundary}\mid \cO)
    \end{align*}
    Let \(\cS_B\) be the set of all possible cut-sets separating \(v\) from \(G\setminus B_\ell(v)\). Write \(\cE_{\F_\ord,\cS}\) for the event that \(\cS\) is a wired boundary in \(\F_\ord\) of \(v\) in \(B_\ell(v)\). From monotonicity, conditional on \(\cO\), \(\cE_{\F_\ord,\cS}\) implies that \(\cS\) is also a wired boundary of \(v\) for \(\F^+\). By summing over these events we get that \(\Pr(e\not\in\F_\ord, e\in \F^+,\,\F_\ord\text{ has a wired boundary}\mid\cO) \) is at most
    \[
    \sum_{\cS\in\cS_B} \Pr(e\not\in\F_\ord, e\in\F^+ \mid \cE_{\F_\ord,\cS},\cO)\Pr(\cE_{\F_\ord,\cS}\mid\cO)
    \]
    To finish the proof, we show that for each \(\cS\in\cS_B\), \(\Pr(\F_\ord\cap\cN_v\neq\F^+\cap\cN_v \mid \cE_{\F_\ord,\cS}, \cO) = 0\).

    Conditional on \(\cE_{\F_\ord,\cS},\cO\), we can sample  \(\F_\ord\), \(\F^+\) by first generating the edges outside of the component of \(v\) in \(G\setminus \cS\) -- call these edges \(E_O\) -- and then, conditional on \(E_O\) and \(\cE_{\F_\ord,\cS}\) the remainder of the edges. Note that for \(\cS\in\cS_B\), the events ``\(\cS\) is a wired boundary in \(\F_\ord\)'' (which implies it is a wired boundary also in \(\F^+\)) and \(\cO\) are determined by \(\F_\ord\cap E_O\).

    Thus, conditional on \(\cE_{\F_\ord,\cS}\), \(\cO\), and then any (possible) subconfiguration of \((\F_\ord, \F^+)\) on \(E_O\), they both have a wired boundary \(\cS\). Then by \cite[Observation 51]{RCrandom}, the marginals edges in \(E\setminus E_O\) agree, hence since \(e\not\in E_O\) and by the fact that \((\F_\ord, \F_+)\) were perfectly coupled, the probability of disagreeing on \(e\) conditional on \(\cE_{\F_\ord,\cS}\) and \(\cO\) is zero.
\end{proof}

\begin{LemmaWBfromgiantpaths}
\stateWBfromgiantpaths
\end{LemmaWBfromgiantpaths}
\begin{proof}[Proof of Lemma~\ref{lem:no-long-path-outside-giant=>wired}]
    Consider the set of paths \(\cP\) of length $\ell$ in  \(G\) that start in \(v\). 

    By assumption, each of these paths have a vertex in the largest component of \(G_{v,F}\), denoted \(\cC_1\). For each path in \(\cP\), consider the \emph{first} vertex on the path belonging to \(\cC_1\) (we think of the path as starting in \(v\)). Let \(\cS\) be the set of these vertices. We next show that \(\cS\) is a wired boundary around \(v\) in \(B_\ell(v)\).

    We can further assume \(\cS\) is not empty: if  $\cS$ were empty, then   \(\cP = \emptyset\), which only happens if \(v\) is not connected to \(G\setminus B_\ell(v)\), and thus an empty set is a trivial wired boundary.

    \noindent{\bf Claim 1.} \(\cS\subseteq V(B_\ell(v))\setminus\{v\}\) is a cut-set of \(G\) separating \(v\) from \(V\setminus B_\ell(v)\)

\noindent{\bf Proof of Claim 1.} 
The paths in $\cP$ are contained in the ball $B_\ell(v)$. Also, \(v\) is not in the graph $G_{v,F}$, hence \(v\not\in\cS\).
Consider any path from \(v\) to \(V\setminus V(B_\ell(v))\) in \(G\), \(v=w_0,w_1,\dots,w_r,\dots,w_{\ell'}\not\in B_\ell(v)\). Clearly, \(w_0,\dots,w_\ell\) is a path in \(\cP\) and thus there is \(1\leq i\leq \ell\) such that \(w_i\in\cS\). Therefore any path from \(v\) to the outside of the ball is disconnected in \(G\setminus\cS\),  proving Claim~1.

    Let \(\cC_v\) be the component of \(v\) in \(G\setminus\cS\). To conclude that \(\cS\) is a wired boundary, we need to show vertices of \(\cS\) belong to a single component in \((V\setminus V(\cC_{v}), \{e \in F : e\cap V(\cC_v) = \emptyset\})\). In particular, we show that this component is  \(\cC_1\).

    \noindent{\bf Claim 2.} No vertices of \(\cC_1\) belong to \(\cC_v\).

\noindent{\bf Proof of Claim 2.} 
    Suppose for contradiction \(u\in V(\cC_v)\cap V(\cC_1)\). By  Claim~1, \(\cC_v\) is a subgraph of \(B_\ell(v)\), and thus there exists a path \(v = w_0,w_1,\dots,w_{\ell'} = u\) of length \(\ell'\leq \ell\) from \(v\) to \(u\) in \(\cC_v\). Since we assumed \(\cS\) is not empty and \(\cS\subseteq V(\cC_1)\), wlog we can extend this path to be length \(\ell\), containing such \(u\), say to \(w_0,w_1,\dots,w_{\ell'}=u,w_{\ell'+1},\dots,w_\ell\). But this is a path in \(\cP\). By the choice of \(\cS\) and \(u\), there is \(1\leq i\leq \ell'\) such that \(w_i\in \cS\). But then \(w_i\not\in \cC_v\), which contradicts that \(w_0,\dots,w_{\ell'}\) is in \(\cC_v\).
Thus no such \(u\) can exist and Claim~2 follows.

Claim~2 implies that \(\cC_1\)  is connected in \((V\setminus V(\cC_{v}), \{e \in F : e\cap V(\cC_v) = \emptyset\})\), thus in particular all vertices of \(\cS\) belong to it in this graph as well, so these vertices are all in the same component.
    
     Thus it follows that \(\cS\) is a wired boundary around \(v\) in \(B_\ell(v)\).
\end{proof}

\section{Proof of Theorems~\ref{thm:RCdis} and~\ref{thm:RCord}}\label{sec:proofmain}
In this section, we prove Theorems~\ref{thm:RCdis} and~\ref{thm:RCord}. The proofs of these follow closely those  in \cite{RCrandom} (which are in turn similar to previous arguments in \cite{SinclairsGheissari2022,BG}), so we outline the key ideas and only included detail on the probability amplification (which is not included therein).
\begin{thmRCdis}
\statethmRCdis
\end{thmRCdis}
\begin{proof}
First, we have that, whp \(G\sim\cG_{n,d}\), satisfies:
\begin{enumerate}
\item Corollary~\ref{cor:boundary-of-phases-are-unlikely}  so that, for $\F_\dis \sim \pi_{G,\dis}$, 
$\Pr(|\F_\dis| \geq 
n m_d(\beta_c) + \delta n) = \emm^{-\Omega(n)}$, 
\item Theorem~\ref{thm:wsm-disordered}, i.e., there exists $\delta>0$ so that \(G\) has WSM within the disordered phase at a distance \(\ell\leq(\tfrac{1}{2}-\delta)\log_{d-1}n\).
\end{enumerate}
Also, for $\ell$ as above, there exists a constant $L>0$ such that for every vertex $v$, $B_{\ell}(v)$ is $L$-treelike, see, e.g., \cite[Fact 2.3.]{BG}. Consider any graph $G$ that satisfies these properties.

We next outline the theorem  for \(\varepsilon = \tfrac{1}{4}\), the full details for this part can be found in \cite{RCrandom}. Then,  we prove that probability amplification applies for any \(\varepsilon=\emm^{-\Omega(n)}\).  Consider two copies of the random-cluster dynamics: (1) \((X_t)_{t\geq 0}\) starting from all-out, and (2) \((\hat X_t)_{t\geq 0}\) starting from \(\pi_{G,\dis}\) which is restricted to the disordered phase by ignoring updates that would make the configuration not disordered. Note that the stationary distribution of \(\hat X_t\) is \(\pi_{G,\dis}\), and thus for all \(t\geq 0\), \(\hat X_t\sim\pi_{G,\dis}\). 
    
    We couple the two chains by choosing the same edge in each step, and same \(U[0,1]\) random variable to decide whether to keep the edge in the next configuration. Note that if \(\hat X_t\) has not ignored any updates by the time \(t\), by monotonicity \(\hat X_t \supseteq X_t\).     Since \(\hat X_t\sim\pi_{G,\dis}\), we can bound \(\TV{X_T-\pi_{G,\dis}}\leq\Pr(X_T\neq\hat X_T)\), which will next show is at most \(\tfrac{1}{4}\) for \(T = Cn\log n\) for some constant \(C > 0\).
    Let \(\cE_{<t}\) be the event that no transitions have been ignored by time \(t\). Since we assumed \(G\) is such that Corollary~\ref{cor:boundary-of-phases-are-unlikely} applies, the probability that \(\hat X_t\) is in a state from which it is possible to leave the disordered phase is \(\emm^{-\Omega(n)}\), thus by the union bound \(\Pr(\overline{\cE_{<T}}) = T\emm^{-\Omega(n)}\).   
    
    Next, we upper bound $\Pr(X_T\neq \hat X_T)$ by summing over disagreements on individual edges. We have in particular that
    \begin{align}
    \Pr(X_T\neq \hat X_T)\leq&  m\Pr(\overline{\cE_{<T}}) + \sum_{e\in E(G)} \Pr(\mathds{1}\{e\in X_T\}\neq \mathds{1}\{e\in \hat X_T\} \mid \cE_{<T}). \label{eq:tv-bound-by-ub}
    \end{align}
    Fix an arbitrary edge \(e\) and let \(v\) be a vertex incident to \(e\). We will bound \(\Pr(\mathds{1}\{e\in X_T\}\neq \mathds{1}\{e\in \hat X_T\} \mid \cE_{<T})\) by considering another random-cluster dynamics \((X_t^v)_{t\geq 0}\), which starts from all-out and for which where we restrict all updates to be inside the ball \(B_\ell(v)\). Formally,  \(X_t^v\) is coupled with \(X_t\) and \(\hat X_t\) using the same coupling as above, ignoring updates on edges outside the ball \(B_\ell(v)\). By monotonicity, conditional on \(\cE_{<T}\), \(X^v_T \subseteq X_T \subseteq \hat X_T\), thus also \(\Pr(\mathds{1}\{e\in X_T\}\neq \mathds{1}\{e\in \hat X_T\} \mid \cE_{<T})\leq \Pr(\mathds{1}\{e\in X^v_T\}\neq \mathds{1}\{e\in \hat X_T\} \mid \cE_{<T})\leq \Pr(e\in X^v_T\mid \cE_{<T}) - \Pr(e\in\hat X_T \mid \cE_{<T})\), where the last inequality follows by monotonicity. Using the triangle inequality on  the last bound, and the inequality  \(|\Pr(A)-\Pr(A\mid B)|\leq 2\Pr(\overline B)\) to remove the conditioning, we obtain that
    \[
   \Pr(\mathds{1}\{e\in X_T\}\neq \mathds{1}\{e\in \hat X_T\} \mid \cE_{<T}) \leq 4 \Pr(\overline{\cE_{<T}}) + |\Pr(e\in X_T^v) - \pi_{G,v, \ell}^-(1_e)| + |\pi_{G,v, \ell}^-(1_e) - \Pr(e\in \hat X_T)|,
    \]
where $\pi^-_{G,v,\ell}$ is the 
distribution on $B_\ell(v)$ conditioned on an all-out boundary. 
    Since \(\Pr(e\in\hat X_T) = \pi_{G,\dis}(1_e)\), WSM within the disordered phase bounds the third term by \(\frac{1}{100m}\).
     The second term is also bounded by $\frac{1}{100m}$  using mixing time estimates, in this case log-Sobolev inequalities. We don't give the full details here (which, as mentioned before, can also be found in \cite{RCrandom, SinclairsGheissari2022,BG}) but the idea is that, since the neighbourhood of $v$ is treelike (up to removing a constant number of edges), the log-Sobolev constant for $\pi_{G,v, \ell}^-(1_e)$ is a constant factor away from that on the regular tree. This in turn is just percolation so the log-Sobolev constant for the chain \(X_v^t\) is at most \(C_\dis / |V(B_\ell(v))|\) (in the ordered regime, the bound for the log-Sobolev constant on the tree with + boundary condition is given in \cite{treeRC}). Since \(\pi_{G,v, \ell}^-\) is the stationary distribution of the chain \(X_v^t\), the term $|\Pr(e\in X_T^v) - \pi_{G,v, \ell}^-(1_e)|$ decays as $n\emm^{-C_\dis T_v / 3}$  where $T_v$ is the number of updates inside the ball $B_{\ell}(v)$ by time $t$, which ultimately gives the claimed bound $\frac{1}{100m}$.

     Combining the above and going back to \eqref{eq:tv-bound-by-ub}, we obtain
\[
    \Pr(X_T\neq \hat X_T) \leq 5Cmn\log n \emm^{-\Omega(n)} + \frac{m}{50m} \leq 1/4
    \]
for all \(n\) large enough. This concludes the proof for \(\varepsilon = 1/4\). 

In the remainder of the proof, we denote the \(T\) established above for $\eps=1/4$ by \(T_{1/4}\).   
    To get the claimed TV-distance bound for all \(\varepsilon\geq \emm^{-\Omega(n)}\), we modify the standard probability amplification argument using monotonicity. Formally, we prove the following inequality for all integer \(k\geq 0\), whenever \(T' \geq kT_{1/4}\) (the second inequality follows from induction and the definition of \(T_{1/4}\)).    
    \begin{align}
        \Pr(X_{kT_{1/4}} \neq \hat X_{kT_{1/4}}\mid\cE_{<T'}) &\leq \frac{\Pr(X_{(k-1)T_{1/4}}\neq\hat X_{(k-1)T_{1/4}}\mid\cE_{<T'})\Pr(X_{T_{1/4}}\neq\hat X_{T_{1/4}})}{\Pr(\cE_{<T'})} \label{eq:tv-decay-induction-step} \\
        &\leq [4\Pr(\cE_{<T'})]^{-k}\label{eq:conditional-tv-decay}
    \end{align}

    We introduce another copy of coupled Glauber dynamics, \((X'_t)_{t\geq T}\) starting from all-out at the time \(T_{(k-1)T_{1/4}}\), which is to say that that \(t\)-th step of \(X'_t\) is coupled with the \(((k-1)T_{1/4}+t)\)th step of \(X_t\) and \(\hat X_t\) (using the standard monotone coupling). Now, we use two facts: first \(\hat X_{(k-1)T_{1/4}}\sim \pi_{G,\dis}\), \(X'_t\sim X_t\), and thus the probability of \(\hat X_{(k-1)T_{1/4} + t}\neq X'_t\) is the same as the probability of \(X_t\neq \hat X_t\). Second, that conditional on \(\cE_{<T'}\) (for any \(T' \geq t + (k-1)T_{1/4}\)), \(X'_t\subseteq X_{(k-1)T_{1/4}} \subseteq \hat X_{(k-1)T_{1/4}}\) and hence \(X_{(k-1)T_{1/4}}\neq\hat X_{(k-1)T_{1/4}}\implies X'_T\neq \hat X_{(k-1)T_{1/4}}\).

    Thus, conditionally on \(\cE_{<T'}\), it must be the case that both \(X_t\), \(\hat X_t\) has not coupled by time \((k-1)T_{1/4}\) (if they have, by the conditioning they would not disagree before time \(T'\)), and also that \(\hat X_{(k-1)T_{1/4} + t}\), \(X'_t\) has not coupled by time \(t = T_{1/4}\), hence l.h.s. of~\eqref{eq:tv-decay-induction-step} is at most
    \[
    \Pr(X'_{T_{1/4}}\neq\hat X_{kT_{1/4}}\mid \cE_{<T'})\Pr(X_{(k-1)T_{1/4}}\neq\hat X_{(k-1)T_{1/4}}\mid\cE_{<T'}),
    \]
    We conclude~\eqref{eq:tv-decay-induction-step} by noting \(\Pr(X'_{T_{1/4}}\neq\hat X_{kT_{1/4}} \mid \cE_{<T'}) \leq \frac{\Pr(X'_{T_{1/4}}\neq\hat X_{kT_{1/4}})}{\Pr(\cE_{<T'})} = \frac{\Pr(X_{T_{1/4}}\neq\hat X_{T_{1/4}})}{\Pr(\cE_{<T'})}\).

    To finish, recall that \(\Pr(\overline{\cE_{<T'}})\leq T' \emm^{-\Omega(n)}\) and \(T_{1/4} = O(n\log n)\), thus there exists \(\varepsilon_0 = \emm^{-\Omega(n)}\) such that \(T_{1/4}\lceil\log_2(\varepsilon_0/2)\rceil\leq\tfrac{\varepsilon_0}{2}\leq\tfrac 12\). Then, by~\eqref{eq:conditional-tv-decay}, for all \(\varepsilon\in(0,\varepsilon_0)\), 
    \begin{align*}
    \Pr(X_{T_{1/4}\lceil\log_2(2/\varepsilon)\rceil}\neq\hat X_{ T_{1/4}\lceil\log_2(2/\varepsilon)\rceil} )&\leq \Pr(\overline{\cE_{<T'}}) +  \Pr(X_{T_{1/4}\lceil\log_2(2/\varepsilon)\rceil}\neq\hat X_{ T_{1/4}\lceil\log_2(2/\varepsilon)\rceil} \mid \cE_{<T'}) \\&\leq \frac{\varepsilon_0}{2} + 2^{-\lceil\log_2(2/\varepsilon)\rceil}\leq \varepsilon
    \end{align*}
    Thus, for any \(\varepsilon \geq \emm^{-\Omega(n)}\), \(\TV{X_T-\pi_{G,\dis}}\leq\varepsilon\) for \(T = T_{1/4}\lceil\log_2(1/\varepsilon)\rceil = O(n\log n\log(1/\varepsilon))\).
\end{proof}

\begin{thmRCord}
\statethmRCord
\end{thmRCord}
\begin{proof}
    The proof for ordered phase is completely analogous to the proof for the disordered phase, except we compare to the chain starting from all-out, and the inequalities go the other way.
\end{proof}

\end{appendix}

\end{document}